\documentclass[a4paper,11pt]{article}
\usepackage{latexsym,amssymb,enumerate,amsmath,epsfig,amsthm,dsfont}
\usepackage[margin=1in]{geometry}
\usepackage{setspace,color}
\usepackage{tikz}
\usepackage{floatrow}
\usepackage{graphicx,subfigure}
\usepackage[ruled]{algorithm2e}
\usepackage{algorithmic}
\usepackage{epstopdf}
\usepackage[multidot]{grffile}
\usepackage{comment}
\usepackage{multirow}
\usepackage{hyperref}
\hypersetup{
	colorlinks=true,
	linkcolor=blue,
	citecolor=blue,
	filecolor=blue,      
	urlcolor=blue
}

\newcommand{\bx}{\mathbf{x}}

\newcommand{\bp}{\mathbf{p}}
\newcommand{\bq}{\mathbf{q}}

\newcommand{\bn}{\mathbf{n}}

\newcommand{\ADE}{\mathrm{ADE}}
\newcommand{\CADE}{\mathrm{CADE}}

\DeclareMathOperator*{\argmin}{arg\,min}

\newcommand{\cL}{\mathcal{L}}
\newcommand{\cH}{\mathcal{H}}

\newtheorem{theorem}{Theorem}[section]
\newtheorem{remark}{Remark}[section]

\newtheorem{corollary}{Corollary}[section]

\begin{document}

\title{Fast operator splitting methods for obstacle problems}
\author{Hao Liu \thanks{Department of Mathematics, Hong Kong Baptist University, Kowloon Tong, Hong Kong. (Email: haoliu@hkbu.edu.hk).},
Dong Wang \thanks{School of Science and Engineering, The Chinese University of Hong Kong, Shenzhen. 
	(Email: wangdong@cuhk.edu.cn). The work of Dong Wang is supported by Guangdong Provincial Key Laboratory of Big Data Computing, The Chinese University of Hong Kong, Shenzhen. D. Wang acknowledges support from National Natural Science Foundation of China grant 12101524 and the University Development Fund from The Chinese University of Hong Kong, Shenzhen (UDF01001803).}}

\date{}

\maketitle

\begin{abstract}
The obstacle problem is a class of free boundary problems which finds applications in many disciplines such as porous media, financial mathematics and optimal control.
In this paper, we propose two operator--splitting methods to solve the linear and nonlinear obstacle problems.
The proposed methods have three ingredients: (i) Utilize an indicator function to formularize the constrained problem as an unconstrained problem, and associate it to an initial value problem. The obstacle problem is then converted to solving for the steady state solution of an initial value problem. (ii) An operator--splitting strategy to time discretize the initial value problem. After splitting,  a heat equation with obstacles is solved and other subproblems either have explicit solutions or can be solved efficiently. (iii) A new constrained alternating direction explicit method, a fully explicit method, to solve heat equations with obstacles. The proposed methods are easy to implement, do not require to solve any linear systems and are more efficient than existing numerical methods while keeping similar accuracy. Extensions of the proposed methods to related free boundary problems are also discussed.
\end{abstract}

\section{Introduction}
The obstacle problem is a free boundary problem which studies the equilibrium state of an elastic membrane constrained over an obstacle with fixed boundary conditions. It has broad applications in porous media, financial mathematics, elasto-plasticity, and optimal control \cite{attouch2014variational,caffarelli1998obstacle,friedman1988free,rodrigues1987obstacle}. 
The classical obstacle problem is motivated by studying the equilibrium state of an elastic membrane on the top of an obstacle. Given an obstacle $\Psi$ and a boundary condition $g$,  the equilibrium state is solved by minimizing the surface area: 
\begin{align}
	\min_{u} \int_{\Omega} \left(\sqrt{|\nabla u|^2+1}\right) d\bx, \ u\geq \Psi \mbox{ in } \Omega, \ u=g \mbox{ on } \partial\Omega,
	\label{eq.obs}
\end{align}
where $u$ represents the membrane, $\Omega$ denotes the computational domain, and $\partial \Omega$ denotes the boundary of $\Omega$. A well-known variation of (\ref{eq.obs}) is its linearization
\begin{align}
	\min_{u} \int_{\Omega} |\nabla u|^2 d\bx, \ u\geq \Psi \mbox{ in } \Omega, \ u=g \mbox{ on } \partial\Omega,
	\label{eq.obs.linear}
\end{align}
in which the functional is called the Dirichlet energy of $u$. Problem  (\ref{eq.obs}) and (\ref{eq.obs.linear}) are referred to as nonlinear obstacle problem and linear obstacle problem, respectively. For both problems, the contacting set on which $u=\Psi$ is a free boundary problem \cite{friedman1988free,kinderlehrer2000introduction}. Theories on the obstacle problem and the free boundary problem,  such as the regularity of the free boundary, have been studied for decades \cite{caffarelli1998obstacle, caffarelli1976regularity, lindqvist1988regularity, petrosyan2012regularity,athanasopoulos2008structure}. 
The regularity of the solution and free boundary of fractional order of obstacle problems are studied in \cite{caffarelli2008regularity, silvestre2007regularity}. We refer interested readers to \cite{ros2018obstacle} for an overview of classical results on obstacle and free boundary problems.

The broad applications of obstacle problems make efficient numerical solvers in high demand. A lot of efforts have been devoted into this topic. In literature, one class of numerical methods directly tackles the Euler-Lagrange equation of (\ref{eq.obs}) or (\ref{eq.obs.linear}), a variational inequality. An early work is \cite{lions1979splitting} in which a splitting method was proposed. Finite element based methods are proposed in \cite{glowinski2008lectures,hoppe1994adaptive, chen2000residual, kornhuber1994monotone, kornhuber1996monotone,johnson1992adaptive}. Among them, an over--relaxation method with projection was proposed in \cite{glowinski2008lectures}. A multilevel preconditioner was used in \cite{hoppe1994adaptive}. The authors of  \cite{kornhuber1994monotone} designed monotone multigrid methods and \cite{chen2000residual} used positivity preserving interpolation. An adaptive method with a posteriori error estimate was proposed in \cite{johnson1992adaptive}. In \cite{hoppe1987multigrid}, the author proposed an iterated scheme in which each step is solved by a multigrid algorithm. Recently, a simple semi-implicit method was proposed in \cite{liu2020simple} to solve (\ref{eq.obs}) and (\ref{eq.obs.linear}). The authors decomposed the Euler-Lagrange equation into linear and nonlinear parts, which are treated implicitly and explicity in the numerical scheme, respectively. The proposed schemes can be easily extended to solve fractional obstacle problems and partial differential equations with obstacle constratints. The domain decomposition methods for solving variational inequalities are explored in \cite{badea2004convergence,tai2003rate}.

Another line of research relaxes the constraint by a penalty term. A $L^2$ penalty relaxation was proposed in \cite{scholz1984numerical}, in which convergence rate with respect to the penalty weight was studied. In \cite{french2001pointwise}, the authors proposed a $L^2$ penalized finite element method to solve (\ref{eq.obs.linear}). The constraint is incorporated with the objective functional as a Lagrange multiplier in \cite{hintermuller2011obstacle}. Then a smoothed version of the optimization problem is solved numerically.
In \cite{tran20151}, the authors proposed an $L^1$--like penalty to relax the constraint of the linear problem (\ref{eq.obs.linear}). The authors proved that under appropriate conditions, the relaxed problem and the linear problem (\ref{eq.obs.linear}) have the same solution. A split Bregmann algorithm was used to solve the relaxed optimization problem.  The penalty proposed in \cite{tran20151} was then generalized to non--smooth variational problems and degenerate elliptic equations in \cite{schaeffer2018penalty}. The authors of \cite{zosso2017efficient} proposed primal--dual methods to solve the constrained problem (\ref{eq.obs}) and (\ref{eq.obs.linear}), and their relaxed form with the $L^1$ penalty introduced in \cite{tran20151}. A modified level set method was proposed in \cite{majava2004level} to solve problem (\ref{eq.obs.linear}). The authors first formularize the problem into a nonsmooth optimization problem. Then a proximal bundle method was used to solve the nonsmooth problem and a gradient method was proposed to solve the regularized problem. 

In this paper, we start with the optimization problem (\ref{eq.obs}) and (\ref{eq.obs.linear}), and incorporate the operator splitting method with the alternating direction explicit (ADE) method to design fast solvers. 
The operator splitting method is a large class of methods which decomposes a complicated problem into several easy--to--solve subproblems. Such a method is powerful in solving complicated optimization problems and differential equations. It has been successfully applied in many fields, such as image processing \cite{deng2019new,liu2021color,he2020curvature}, inverse problem \cite{glowinski2015penalization}, numerical PDEs \cite{glowinski2019finite, liu2019finite} and fluid-structure interactions \cite{bukavc2013fluid}. We refer readers to \cite{glowinski2017splitting,glowinski2016some} for a complete discussion on operator splitting methods. Many efficient and famous algorithms can be referred to operator splitting into the class of operator splitting methods, such as threshold dynamics method \cite{Ma_2021, wang2019iterative,Wang_2021,Wang_2017}, Strang operator splitting \cite{li2022operator,li2022stability,li2022stability2}, and so on. A special operator splitting method is the ADE method. ADE was first introduced and studied in \cite{barakat1966solution,larkin1964some} for the numerical solution of heat equations, and was  extended to solve nonlinear evolution equations in \cite{leung2005alternating}. The aforementioned works only considered static Dirichlet boundary conditions. Recently, ADE for time evolution equations with the time-dependent Dirichlet boundary condition and the Neumann boundary condition was studied in \cite{liu2019alternating}. Under appropriate conditions, it has been shown that ADE is second order accurate in time and unconditionally stable for linear time evolution PDEs, similar to the properties of the alternating direction implicit method (ADI). The benefits of ADE over ADI is that ADE is a fully explicit scheme. Therefore no linear system needs to be solved.

In this paper, we propose operator splitting methods for problem (\ref{eq.obs}) and (\ref{eq.obs.linear}). By introducing an indicator function, we first rewrite the constrained problem into a nonsmooth unconstrained optimization problem. Then the  problem can be formulated into solving an initial value problem whose steady state solution is the minimizer of the optimization problem. We use the operator--splitting method to time discretize the initial value problem and decompose it into several subproblems. After time--discretization, one subproblem is equivalent to solving a constrained heat equation, for which we modify ADE and propose a constrained ADE method. Other subproblems either have explicit solvers or can be solved efficiently. Extensions of our method to solve other free boundary problems, such as the double obstacle problem and the two-phase membrane problem, are also discussed. Our methods are easy to implement and are more efficient than many existing methods.

We structure this paper as follows: In Section \ref{sec.reformulation}, we reformularize both the linear and nonlinear obstacle problems as unconstrained optimization problems. We present our operator--splitting methods in Section \ref{sec.operator} and the proposed constrained alternating direction explicit method in Section \ref{sec.numericalDis}. Extensions of the proposed methods to other free boundary problems are discussed in Section \ref{sec.extension}. We demonstrated the effectiveness of the proposed methods in Section \ref{sec.experiments} and conclude this paper in Section \ref{sec.conclusion}.

%

 \section{Reformulations of obstacle problems}
 \label{sec.reformulation}
 We present in this section our reformulations of the linear obstacle problem and nonlinear obstacle problem. The idea is motivated by \cite{deng2019new, liu2021color} and both problems will be reformulated to unconstrained optimization problems by utilizing indicator functions. 
 \subsection{Preliminary}
 We define some quantities that will be used in the subsequent sections. 
 
Let $\Omega\subset \mathbb{R}^D$ be the computational domain, and $\psi$ be the obstacle defined on it. Denote the functional in the linear problem (\ref{eq.obs.linear}) and the nonlinear problem (\ref{eq.obs}) by $E_1$ and $E_2$:
 \begin{align}
 	E_1(u)=\int_{\Omega} \left(|\nabla u|^2-fu \right)d\bx, \quad 	E_2(u)=\int_{\Omega} \left(\sqrt{|\nabla u|^2+1}-fu\right) d\bx.
 	\label{eq.obs.energy}
 \end{align}
One can show that if $u$ is a minimizer of $E_1$ or $E_2$, the optimality condition for $u$ reads as
\begin{align}
	\min\{-\nabla^2 u-f,u-\psi\}=0.
	\label{eq.obs1.pde}
\end{align}
or 
\begin{align}
	\min\left\{-\nabla\cdot \frac{\nabla u}{\sqrt{|\nabla u|^2+1}} -f,u-\psi\right\}=0.
\end{align}

In the formulation, we relax the constraint $u\geq \psi$ by introducing an indicator function. We first define the set on which the constraint is satisfied:
$$
\Sigma_{\psi}=\{u\in \cL^2(\Omega):u\geq \psi\},
$$
where $\cL^2(\Omega)$ denotes the $L^2$ space over $\Omega$.
Define the indicator function $I_{\Sigma_{\psi}}(u)$ as:
$$
I_{\Sigma_{\psi}}(u)=\begin{cases}
	0, & \mbox{if } u\in \Sigma_{\psi} \\
	\infty, & \mbox{otherwise}.
\end{cases}
$$
For any function $u$, we have $I_{\Sigma_{\psi}}(u)=0$ if $u\geq \psi$, and $I_{\Sigma_{\psi}}(u)=\infty$ otherwise. Incorporate the functional $I_{\Sigma_{\psi}}(u)$ into $E_1$ or $E_2$, the constraint will be enforced by minimizing the new functional.
 \subsection{Reformulation of the linear obstacle problem }
Denote the functional in $E_1$ by
$$
J_1(u)=\int_{\Omega} \left(|\nabla u|^2-fu\right) d\bx.
$$
We reformulate (\ref{eq.obs}) as
 \begin{align}
   \min_{v\in \cH_g^1(\Omega)} \left[J_1(v)+I_{\Sigma_{\psi}}(v)\right],
    \label{eq.LinearObs.op}
 \end{align}
where $\cH_g^1(\Omega)$ is the Sobolev space defined as
$$
\cH^1_g(\Omega)=\left\{u: u\in \cH^1(\Omega), u=g \mbox{ on } \partial\Omega\right\}$$
with  
$$\cH^1(\Omega)=\left\{u: u\in \cL^2(\Omega), \nabla u\in (\cL^2(\Omega))^D\right\}.
$$
As the constraint is enforced by minimizing the second term, the functional in (\ref{eq.LinearObs.op}) has the same minimizer as (\ref{eq.obs}).
Problem (\ref{eq.LinearObs.op}) can be solved easily by operator-splitting methods since each term is a simple function of $u$, as will be discussed in Section \ref{sec.operator}.

\subsection{Reformulation of the nonlinear problem}
The nonlinear problem (\ref{eq.obs}) is more complicated than (\ref{eq.obs.linear}). The nonlinearity makes it difficult to formulate (\ref{eq.obs}) as a sum of simple functions of $u$, like the form of (\ref{eq.LinearObs.op}). To decouple the nonlinearity, we introduce a vector valued variable $\bp$ and consider the following constrained problem
 \begin{align}
	\begin{cases}
		(u,\bp)=\argmin\limits_{u,\bp} \displaystyle \int_{\Omega} \left(\sqrt{|\bp|^2+1}-fu\right) d\bx,\\
		\bp=\nabla u,\\
		u\geq \psi.
	\end{cases}
	\label{eq.NonLinearObs.enerngy1}
\end{align}

We add the indicator function $I_{\Sigma_{\psi}}(u)$ to the functional to drop the constraint $u\geq \psi$, and use the penalty term $|\bp-\nabla u|_2^2$ to relax the constraint $\bp=\nabla u$. The resulting problem becomes
\begin{align}
	(u,\bp)=\argmin\limits_{v\in \cH_g^1(\Omega),\bp\in \left(\cL^2(\Omega)\right)^D} J_2(v,\bp)+J_3(v,\bp) +\int_\Omega I_{ \Sigma_{\psi}}(v) \ d\bx,
	\label{eq.NonLinearObs.op}
\end{align}
where 
$$
J_2(v,\bp)=\int_{\Omega}\left(\sqrt{|\bp|^2+1}-fv\right) d\bx,\ J_3(v,\bp)=\frac{\alpha}{2}\int_{\Omega} |\bp-\nabla v|^2 d\bx
$$
and $\alpha>0$ is a weight parameter penalizing the mismatch between $u$ and $\bp$. 

\begin{remark}
	In the formulation (\ref{eq.NonLinearObs.op}), we adopt the penalty $|\bp-\nabla u|_2^2$ so that by using operator splitting method to solve (\ref{eq.NonLinearObs.op}), one substep is equivalent to solving a constrained heat equation. Such a problem can be solved robustly and efficiently by the proposed constrained ADE method, see Section \ref{sec.op.nonlinear} and \ref{sec.CADE} for details. 
\end{remark}
 
 \section{Operator-splitting methods}\label{sec.operator}
 \subsection{An operator splitting method for the linear problem}
 The Euler--Lagrange equation of (\ref{eq.LinearObs.op}) reads as
 \begin{align}
 	D_u J_1+\partial_u  I_{\Sigma_{\psi}}\ni 0,
 	\label{eq.LinearObs.EL}
 \end{align}
where $D_u$ (resp. $\partial_u$) denotes the differential (resp. sub-differential) of a differentiable (resp. non-differentiable) function with respect to $u$. We associate the optimality condition (\ref{eq.LinearObs.EL}) with the initial value problem (dynamical flow)
 \begin{align}
 	\frac{\partial u}{\partial t}+D_u J_1+\partial_u  I_{\Sigma_{\psi}}\ni 0,
 	\label{eq.LinearObs.ivp}
 \end{align}
Note that the steady state solution of (\ref{eq.LinearObs.ivp}) solves (\ref{eq.LinearObs.op}).
To time-discretize (\ref{eq.LinearObs.ivp}), we adopt the Lie scheme (see \cite{glowinski2017splitting,glowinski2016some} and the reference therein). Let $t^n=n\tau$ for some time step $\Delta t$. Denote the initial condition by $u_0$, we update $u^{n+1}$ by the following two steps:\\
 \emph{Initialization}:
\begin{align}
 u^0=u_0.
\end{align}
 \emph{Fractional step 1}: Solve
\begin{align}
\begin{cases}
 \frac{\partial u}{\partial t}+D_u J_1(u)=0 \mbox{ on } \Omega\times  (t^n,t^{n+1}),\\
 u(t^n)=u^n,
 \end{cases}
 \label{eq.LinearObs.op.1}
 \end{align}
 and set $u^{n+1/2}=u(t^{n+1})$.

 \noindent\emph{Fractional step 2}: Solve
 \begin{align}
 \begin{cases}
   \frac{\partial u}{\partial t}+\partial_u I_{\Sigma_{\psi}}(u)\ni 0 \mbox{ on } \Omega\times (t^n,t^{n+1}),\\
   u(t^n)=u^{n+1/2},
 \end{cases}
  \label{eq.LinearObs.op.2}
 \end{align}
 and set $u^{n+1}=u(t^{n+1})$.
 
 The scheme (\ref{eq.LinearObs.op.1})--(\ref{eq.LinearObs.op.2}) is only semi-constructive since we still need to solve two subproblems. Recall that $D_u J(u)=-\nabla^2 u-f$. Solving (\ref{eq.LinearObs.op.1}) is equivalent to solving the heat equation
 \begin{align}
 	u_t-\nabla^2 u=f,
 \label{eq.lap}
 \end{align}
 for which we use the alternating direction explicit (ADE) method. We propose to time discretize (\ref{eq.LinearObs.op.1})--(\ref{eq.LinearObs.op.2}) by an  explicit scheme.
 For $n>0$, we update $u^n\rightarrow u^{n+1/2}\rightarrow u^{n+1}$ as follows:
 \begin{align}
 	\begin{cases}
 		u^{n+1/2}=\ADE(\nabla^2,u^{n+1/2},u^n,f,\Delta t),\\
 		u^{n+1}=\max\{u^{n+1/2},\psi\},
 	\end{cases}
 	\label{eq.LinearObs.ADE}
 \end{align}
 where $\ADE$ denote the ADE solver, which is introduced in Section \ref{sec.ADE}.
  
 In (\ref{eq.LinearObs.ADE}), the second step is only a thresholding step. Since ADE is a fully explicit scheme, we can easily incorporate it with the thresholding. We will propose a constrained ADE (CADE) scheme to solve the two steps in (\ref{eq.LinearObs.ADE}) together. Then the updating formula for $u^{n+1}$ reduces to a one-step scheme
  \begin{align}
 		u^{n+1}=\CADE(\nabla^2,u^{n+1},u^n,f,\psi,\Delta t),
 	\label{eq.LinearObs.CADE}
 \end{align}
 where $\CADE$ represents the constrained ADE method which will be discussed in Section \ref{sec.CADE}.
 
  \subsection{An operator-splitting method for the nonlinear problem }
  \label{sec.op.nonlinear}
%
The optimality condition of (\ref{eq.NonLinearObs.op}) is 
 \begin{align}
	\begin{cases}
		D_{\bp} J_2(\bp,u)+D_{\bp} J_3(\bp,u)=0,\\
		D_{u} J_2(\bp,u)+D_{u} J_3(\bp,u)+\partial_u I_{\Sigma_{\psi}}\ni 0.
	\end{cases}
\label{eq.NonLinearObs.op1}
\end{align}
To solve (\ref{eq.NonLinearObs.op1}), we associate it with the following initial value problem
 \begin{align}
 	\begin{cases}
 		\frac{\partial \bp}{\partial t}+D_{\bp} J_2(\bp,u)+D_{\bp} J_3(\bp,u)=0,\\
 		\gamma\frac{\partial u}{\partial t}+D_{u} J_2(\bp,u)+D_{u} J_3(\bp,u)+\partial_u I_{\Sigma_{\psi}}\ni 0,
 	\end{cases}
 	\label{eq.NonLinearObs.ivp}
 \end{align}
 where $\gamma(>0)$ is a parameter controlling the evolution speed of $u$.
 
 It's straightforward to see that if $(u,\bp)$ is the steady state of (\ref{eq.NonLinearObs.ivp}), then $u$ is a minimizer of (\ref{eq.NonLinearObs.op}), and an approximation of the minimizer of  (\ref{eq.obs}).
 Let $(u_0,\bp_0)$ be an initial condition.
 We use the following Lie scheme to solve (\ref{eq.NonLinearObs.ivp}) for the steady state:\\
 \emph{Initialization}:
 \begin{align}
 	(u^0,\bp^0)=(u_0,\bp_0).
 \end{align}
 \emph{Fractional Step 1}: Solve
 \begin{align}
 	\begin{cases}
 		\frac{\partial \bp}{\partial t}+D_{\bp} J_2(u,\bp)=\mathbf{0} \\
 		\frac{\partial u}{\partial t}=0\\
 		(u(t^n),\bp(t^n))=(u^n,\bp^n)
 	\end{cases}
 	\mbox{ on } \Omega\times (t^n,t^{n+1}),
 	\label{eq.NonLinearObs.op.1}
 \end{align}
 and set $u^{n+1/3}=u(t^{n+1}),\bp^{n+1/3}=\bp(t^{n+1})$.\\
 \emph{Fractional Step 2}: Solve
 \begin{align}
 	\begin{cases}
 		\frac{\partial \bp}{\partial t}=\mathbf{0} \\
 		\gamma\frac{\partial u}{\partial t}+D_u J_2(u,\bp)+D_u J_3(u,\bp)+\partial_u I_{ \Sigma_{\psi}}(u)=0,\\
 		(u(t^n),\bp(t^n))=(u^{n+1/3},\bp^{n+1/3})
 	\end{cases}
 	\mbox{ on } \Omega\times (t^n,t^{n+1}),
 	\label{eq.NonLinearObs.op.2}
 \end{align}
 and set $u^{n+2/3}=u(t^{n+1}),\bp^{n+2/3}=\bp(t^{n+1})$.\\
 \emph{Fractional Step 3}: Solve
 \begin{align}
 	\begin{cases}
 		\frac{\partial \bp}{\partial t}+D_{\bp} J_3(u,\bp)=\mathbf{0} \\
 		\frac{\partial u}{\partial t}=0\\
 		(u(t^n),\bp(t^n))=(u^{n+2/3},\bp^{n+2/3})
 	\end{cases}
 	\mbox{ on } \Omega\times (t^n,t^{n+1}),
 	\label{eq.NonLinearObs.op.3}
 \end{align}
 and set $u^{n+1}=u(t^{n+1}),\bp^{n+1}=\bp(t^{n+1})$. 
 
 The scheme (\ref{eq.NonLinearObs.op.1})--(\ref{eq.NonLinearObs.op.3}) are only semi-constructive. As for \eqref{eq.NonLinearObs.op.3}, one can write
 \begin{align}
 	\begin{cases}
 		\frac{\partial \bp}{\partial t}+\alpha(\bp-\nabla u)=\mathbf{0} \\
 		\frac{\partial u}{\partial t}=0\\
 		(u(t^n),\bp(t^n))=(u^{n+2/3},\bp^{n+2/3})
 	\end{cases}
 	\mbox{ on } \Omega\times (t^n,t^{n+1}),
 	\label{eq.NonLinearObs.op.3.1}
 \end{align}
 whose solution is given by
 \begin{align}
 	\begin{cases}
 		\bp(t^{n+1})=\exp(-\alpha \Delta t)\bp^{n+1/2}+\left(1-\exp(-\alpha \Delta t)\right) \nabla u^{n+1},\\
 		u(t^{n+1})=u^{n+1/2}.
 	\end{cases}
 \end{align}

 For other subproblems, we use the Marchuk-Yanenko scheme \cite{glowinski2017splitting} to time discretize (\ref{eq.NonLinearObs.op.1}), and a semi-implicit scheme to time discretize (\ref{eq.NonLinearObs.op.2}). The updating formulas are given as:\\
 For $n>0$, we update $(u^n,\bp^n)\rightarrow (u^{n+1/2},\bp^{n+1/2}) \rightarrow (u^{n+1},\bp^{n+1})$ as follows:
 \begin{align}
 		&\frac{\bp^{n+1/2}-\bp^n}{\Delta t}+\partial_{\bp} J_2(u^n,\bp^{n+1/2})=\mathbf{0}, \label{eq.NonLinearObs.op.dis.1}\\
 		&\gamma\frac{u^{n+1}-u^n}{\Delta t}+D_u J_2(u^{n+1},\bp^{n+1/2})+D_u J_3(\widetilde{u}^{n+1/2},\bp^{n+1/2})+\partial_u I_{ \Sigma_{\psi}}(u^{n+1})\ni 0, \label{eq.NonLinearObs.op.dis.2}\\
 		&\bp^{n+1}=\exp(-\alpha \Delta t)\bp^{n+1/2}+\left(1-\exp(-\alpha \Delta t)\right) \nabla u^{n+1}, \label{eq.NonLinearObs.op.dis.3}
 \end{align}
where $D_u J_3(\widetilde{u}^{n+1/2},\bp^{n+1/2})$ denotes the computed $D_u J_3$ using $u^n,u^{n+1}$ and $\bp^{n+1/2}$. Such a treatment will lead to a semi-implicit scheme, which is well--suited to be solved by CADE, see Section \ref{sec.NonLinearObs.CADE} for details. In the rest of this section, we discuss numerical solutions to (\ref{eq.NonLinearObs.op.dis.1}) and (\ref{eq.NonLinearObs.op.dis.2}).

 \subsection{On the solution of (\ref{eq.NonLinearObs.op.dis.1})}
 In (\ref{eq.NonLinearObs.op.dis.1}), $\bp^{n+1/2}$ solves
 \begin{align}
 	\bp^{n+1/2}=\argmin_{\bq\in \left(\cL^2(\Omega)\right)^D} \left[ \frac{1}{2}\int_{\Omega} |\bq-\bp^n|^2 d\bx +\Delta t\int_{\Omega} \sqrt{|\bq|^2+1}d\bx \right]
 	\label{eq.NonLinearObs.frac1}
 \end{align}
 By computing the variation of the functional in (\ref{eq.NonLinearObs.frac1}) with respect to $\bq$, $\bp^{n+1/2}$ satisfies
 \begin{align}
 	\bp^{n+1/2}-\bp^n+\Delta t \frac{\bp^{n+1/2}}{\sqrt{1+|\bp^{n+1/2}|^2}}=\mathbf{0}.
 	\label{eq.NonLinear.p.1}
 \end{align}
 We solve (\ref{eq.NonLinear.p.1}) by the fixed point method. Observe that (\ref{eq.NonLinear.p.1}) can be rewritten as
 \begin{align}
 	\left(1+\Delta t \frac{1}{\sqrt{1+|\bp^{n+1/3}|^2}}\right)\bp^{n+1/3}=\bp^n.
 \end{align}

 Set $\bq^0=\bp^n$, we update $\bq^k\rightarrow \bq^{k+1}$ as
 \begin{align}
 	\bq^{k+1}=\left(1+\frac{\Delta t}{\sqrt{1+|\bq^{k}|^2}}\right)^{-1}\bp^n
 	\label{eq.NonLinear.p.2}
 \end{align}
 until $\|\bq^{k+1}-\bq^k\|_{\infty}<\varepsilon_1$ for some small $\varepsilon_1>0$. Denote the converged quantity as $\bq^*$ and set $\bp^{n+1/2}=\bq^*$.
 
 \subsection{On the solution of (\ref{eq.NonLinearObs.op.dis.2})}
 \label{sec.NonLinearObs.CADE}
  In (\ref{eq.NonLinearObs.op.dis.2}), $u^{n+1}$ is the minimizer to
  \begin{align}
  	u^{n+1}=&\argmin_{v\in \cH^1(\Omega)} \bigg[ \frac{\gamma}{2}\int_{\Omega} |v-u^n|^2d\bx- \Delta t \int_{\Omega} fvd\bx  \nonumber\\
  	&\quad +\frac{\alpha}{2} \Delta t\int_{\Omega} |\bp^{n+1/2}-\nabla F(v,u^{n})|^2d\bx + I_{\Sigma_{\psi}}(v)\bigg],
  \end{align}
where $F(u^{n+1},u^n)$ represents $\widetilde{u}^{n+1/2}$. The optimality condition of $u^{n+1}$ is
  \begin{align}
  	\begin{cases}
  		\gamma u^{n+1}-\alpha\Delta t\nabla^2\widetilde{u}^{n+1/2}=\gamma u^n+\Delta tf-\alpha\Delta t\nabla\cdot \bp^{n+1/2},\\
  		u^{n+1}\in \Sigma_{\psi}.
  	\end{cases}
  \label{eq.NonLinearObs.frac2}
  \end{align}
Problem (\ref{eq.NonLinearObs.frac2}) has similar form as (\ref{eq.LinearObs.ivp}) and can be solved by CADE:
\begin{align}
	u^{n+1}=\CADE(\alpha\nabla^2,u^{n+1},u^n,f-\alpha\nabla\cdot \bp^{n+1/2} ,\Psi,\Delta t/\gamma).
	\label{eq.NonLinearObs.CADE}
\end{align}

 \section{Numerical discretization}\label{sec.numericalDis}
 We discuss in this section the numerical discretization of the proposed schemes. Among the three subproblems discussed in Section~\ref{sec.operator}, (\ref{eq.LinearObs.CADE}) and (\ref{eq.NonLinearObs.CADE}) are solved by CADE, (\ref{eq.NonLinear.p.2}) can be solved point-wisely. We will first introduce ADE and then present the proposed CADE. 
 
 \subsection{Alternating direction explicit method}
 \label{sec.ADE}
 ADE is a fully explicit method to solve time-dependent PDEs. In each iteration, ADE carries out separate Gauss--Seidel updates along two directions for one-dimensional problems and along four directions for two-dimensional problems. For linear PDEs under mild conditions, ADE is proved to be unconditionally stable and second order accurate in time \cite{leung2005alternating}. Here we present ADE for one-dimensional heat equation.

Let $\Omega=[0,L]$ be the computational domain discretized by $\{x_i\}_{i=0}^M$ such that $x_i=i\Delta x$ with $\Delta x=L/M$. For any function $v$ defined on $\Omega$, denote $v_i^n=v(t^n,x_i)$ where $t^n=n\Delta t$ for some time step $\Delta t>0$ and further define $v^n=\{v_{i}^n\}_{i=0}^M$. Consider the heat equation
\begin{align}
	\begin{cases}
		u_t-\eta_1\nabla^2 u +\eta_2 u=f \mbox{ in } \Omega,\\
		u(0)=a,\ u(L)=b.
	\end{cases}
\label{eq.heat}
\end{align}

Given the solution $u^n$ at time level $t^n$ and denote $\zeta= \left(1+\frac{\Delta t\eta_1}{\Delta x^2}+\frac{\Delta t\eta_2}{2}\right)^{-1}$, ADE solves for $u^{n+1}$ in the following manner:\\
 \emph{Step 1}: Set 
 \begin{align}
 	u^{n+1,1}=u^{n+1,2}=u^n.
 	\label{eq.ade.1}
 \end{align}
 \emph{Step 2}: For $i=1,2,...,M-1$, compute
 \begin{align}
  u_i^{n+1,1}= \zeta\left( u_i^n+\Delta tf_i^n+\frac{\Delta t\eta_1}{\Delta x^2}\left( u^{n+1,1}_{i-1}-u^n_i+u^n_{i+1}\right)-\frac{\Delta t\eta_2}{2}u^n\right).
 	\label{eq.ade.2.1}
 \end{align}
 For $i=M-1,M-2,...,1$, compute
 \begin{align}
 	u_i^{n+1,2}=\zeta\left( u_i^n+\Delta tf_i^n+\frac{\Delta t\eta_1}{\Delta x^2}\left( u^{n+1,2}_{i+1}-u^n_i+u^n_{i-1}\right)- \frac{\Delta t\eta_2}{2}u^n\right).
 	\label{eq.ade.2.2}
 \end{align}
\emph{Step 3}:
 Compute
 \begin{align}
 	u^{n+1}=\frac{1}{2}\left(u_i^{n+1,1}+u_i^{n+1,2}\right).
 	\label{eq.ade.3}
 \end{align}
 We denote the scheme (\ref{eq.ade.1})--(\ref{eq.ade.3}) by $u^{n+1}=\ADE(\eta_1\nabla^2-\eta_2 I,u^{n+1},u^n,f,\Delta t)$.
 \subsection{Constrained alternating direction explicit method}
 \label{sec.CADE}
 
 (\ref{eq.ade.1})--(\ref{eq.ade.3}) in the ADE method updates the solution on each grid explicitly which gives flexibility to impose constraints to the solution during iterations. Note that problem (\ref{eq.LinearObs.ivp}) and (\ref{eq.NonLinearObs.frac2}) are equivalent to solving the discrete analogue of the following constrained heat equation
 \begin{align}
 	\begin{cases}
 		u_t+\min(-\eta_1\nabla^2u +\eta_2 u-f,u-\psi)=0 \mbox{ in } \Omega,\\
 		u(0)=a,\ u(L)=b.
 	\end{cases}
 \label{eq.cHeat}
 \end{align}
 for some constant $\eta_1>0 $ and $\eta_2\geq 0$.
 Based on ADE, we add a hard thresholding to enforce the constraint at each grid, which leads to CADE:\\
 \emph{Step 1}: Set 
 \begin{align}
 	u^{n+1,1}=u^{n+1,2}=u^n.
 	\label{eq.cade.1}
 \end{align}
 \emph{Step 2}: 
 For $i=1:M-1$, compute
 \begin{align}
 	&u_i^{n+1,1}=\max\bigg\{ \psi_i,\zeta \left( u_i^n+\Delta tf_i^n+\frac{\Delta t\eta_1}{\Delta x^2}\left( u^{n+1,1}_{i-1}-u^n_i+u^n_{i+1}\right)-\frac{\Delta t\eta_2}{2}u^n\right)\bigg\},
 \label{eq.cade.2.1}
 \end{align}
where $\zeta$ is defined under (\ref{eq.heat}).\\
For $i=M-1,1$, compute
 \begin{align}
 	&u_i^{n+1,2}=\max\bigg\{ \psi_i,\zeta \left( u_i^n+\Delta tf_i^n+\frac{\Delta t\eta_1}{\Delta x^2}\left( u^{n+1,2}_{i+1}-u^n_i+u^n_{i-1}\right) - \frac{\Delta t\eta_2}{2}u^n\right)\bigg\}.
\label{eq.cade.2.2}
 \end{align}
 \emph{Step 4}: Compute
 \begin{align}
 	u^{n+1}=\frac{1}{2}\left(u_i^{n+1,1}+u_i^{n+1,2}\right).
 	\label{eq.cade.3}
 \end{align}
 We denote the scheme (\ref{eq.cade.1})--(\ref{eq.cade.3}) as 
 \begin{align}
 	u^{n+1}=\CADE(\eta_1\nabla^2-\eta_2 I,u^{n+1},u^n,f,\psi,\Delta t).
 \end{align}
 
 The following theorem shows that the solution to the discrete analogue of (\ref{eq.cHeat}) is a steady state of scheme (\ref{eq.cade.1})--(\ref{eq.cade.3}):
 \begin{theorem}\label{thm.steady}
 	Let $\eta_1>0$ and $\eta_2\geq0$. For any given obstacle function $\psi$ and source function $f$, let $u^*$ be the solution to the discrete problem
 	\begin{align}
 		\min(-\eta_1\nabla^2_h u+\eta_2 u-f, u-\psi)=0,
 	\end{align}
 	where $\nabla^2_h$ denotes the discretized $\nabla^2$ by the central difference scheme. Then $u^*$ is a steady state of scheme (\ref{eq.cade.1})--(\ref{eq.cade.3}).
 \end{theorem} 
 \begin{proof}
 	We prove the theorem by showing that $u^*$ is the steady state of (\ref{eq.cade.2.1}) and (\ref{eq.cade.2.2}).
 	We decompose $\Omega$ into two parts: $\Omega_1$ and $\Omega_2$ such that $u^*=\psi$ on $\Omega_1$ and $-\nabla_h^2 u^*-f=0$ on $\Omega_2$. Setting $u^n=u^*$, denote the output of (\ref{eq.cade.2.1}), (\ref{eq.cade.2.2}) and (\ref{eq.cade.3}) by $u^{*,1}$, $u^{*,2}$ and $u^{*,*}$, respectively. We first show that $u^*$ is the steady state of (\ref{eq.cade.2.1}), i.e., $u^{*,1}=u^*$, by mathematical induction. Since the Dirichlet boundary condition is used, we have $u^{*,1}_0=u^*_0$. Assume $u^{*,1}_{i-1}=u^*_{i-1}$, we are going to show that $u^{*,1}_{i}=u^*_{i}$. 
 	
 	Define 
 	\begin{align}
 		\widetilde{u}^{*,1}_{i}=\left(1+\frac{\Delta t\eta_1}{\Delta x^2} +\frac{\Delta t\eta_2}{2}\right)^{-1}\left( u_i^*+\Delta tf_i^n+\frac{\Delta t\eta_1}{\Delta x^2}\left( u^{*,1}_{i-1}-u^*_i+u^*_{i+1}\right)-\frac{\Delta t\eta_2}{2}u^*\right).
 		\label{eq.linear.proof.1}
 	\end{align} 
 	We can express $u^{*,1}_i$ as 
 	\begin{align*}
 		u^{*,1}_i=\max(\psi_i,\widetilde{u}^{*,1}_{i}).
 	\end{align*}
 	Rearranging (\ref{eq.linear.proof.1}) gives rise to
 	\begin{align*}
 		\frac{\widetilde{u}^{*,1}_{i}-u^*_i}{\Delta t}&=\frac{\eta_1}{\Delta x^2}\left(u^{*,1}_{i-1}-\widetilde{u}^{*,1}_i-u^{*}_i+u^*_{i+1}\right)-\frac{\eta_2}{2}(\widetilde{u}^{*,1}_{i}+u_i^*)+f_i\\
 		&=\frac{\eta_1}{\Delta x^2}\left(u^*_{i-1}-2u^{*}_i+u^*_{i+1}\right) -\eta_2u_i^*+f_i -\left(\frac{\eta_1}{\Delta x^2}+\frac{\eta_2}{2}\right)(\widetilde{u}^{*,1}_i-u^*_i)\\
 		&=\eta_1(\nabla_h^2 u^*)_i-\eta_2u^*_i+f_i-\left(\frac{\eta_1}{\Delta x^2}+\frac{\eta_2}{2}\right)(\widetilde{u}^{*,1}_i-u^*_i),
 	\end{align*}
 	where the second equality holds since $u^{*,1}_{i-1}=u^*_{i-1}$.
 	Therefore,
 	\begin{align}
 		\left(\frac{1}{\Delta t}+\frac{\eta_1}{\Delta x^2}+\frac{\eta_2}{2}\right)(\widetilde{u}^{*,1}_i-u^*_i)=\eta_1(\nabla_h^2 u^*)_i-\eta_2 u^*_i+f_i.
 		\label{eq.proof.1}
 	\end{align}
 	When $x_i\in \Omega_1$, we have $u^*_i=\psi_i$ and $-\eta_1(\nabla_h^2 u^*)_i+\eta_2 u^*-f_i\geq 0$. Substituting this relation into (\ref{eq.proof.1}) gives rise to
 	\begin{align}
 		\left(\frac{1}{\Delta t}+\frac{\eta_1}{\Delta x^2}+\frac{\eta_2}{2}\right)(\widetilde{u}^{*,1}_i-u^*_i)\leq 0,
 	\end{align}
 	implying that $\widetilde{u}^{*,1}_i\leq u^*_i=\psi_i$. We have
 	\begin{align*}
 		u^{*,1}_i=\max(\psi_i,\widetilde{u}^{*,1}_{i})=\psi_i=u^*_i.
 	\end{align*}
 	When $x_i\in \Omega_2$, we have $u^*_i\geq \psi_i$ and $-\eta_1(\nabla_h^2 u^*)_i+\eta_2 u^*-f_i=0$. Substituting this relation into (\ref{eq.proof.1}) gives rise to
 	\begin{align}
 		\left(\frac{1}{\Delta t}+\frac{\eta_1}{\Delta x^2}+\frac{\eta_2}{2}\right)(\widetilde{u}^{*,1}_i-u^*_i)=0,
 	\end{align}
 	implying that $\widetilde{u}^{*,1}_i=u^*_i$.
 	We have
 	\begin{align*}
 		u^{*,1}_i=\max(\psi_i,\widetilde{u}^{*,1}_{i})=\max(\psi_i,u^*_i)=u^*_i.
 	\end{align*}
 	Therefore, $u^{*,1}_i=u^*_i$, By mathematical induction, we have $u^{*,1}=u^*$. 
 	
 	Similarly, one can show that $u^{*,2}=u^*$ by inducting from $i=M$ to $i=0$. Thus 
 	\begin{align}
 		u^{*,*}=\frac{1}{2}(u^{*,1}+u^{*,2})=u^*,
 	\end{align}
 	and the proof is finished.
 \end{proof}

 Note that the numerical solver (\ref{eq.LinearObs.CADE}) for the linear problem (\ref{eq.obs.linear}) is simply a CADE algorithm. Setting $\eta_1=1,\eta_2=0$ gives the following corollary:
 \begin{corollary}\label{coro.steady}
 	For any given obstacle function $\psi$ and source function $f$, let $u^*$ be the solution to the discrete analogue of the optimality condition for the linear problem (\ref{eq.obs.linear})
 	\begin{align}
 		\min(-\nabla^2_h u-f, u-\psi)=0,
 	\end{align}
 where $\nabla^2_h$ denotes the discretized $\nabla^2$ by the central difference scheme. Then $u^*$ is a steady state of scheme (\ref{eq.LinearObs.CADE}).
 \end{corollary} 

  \subsubsection{On the solution of (\ref{eq.NonLinearObs.CADE}) }
  Problem (\ref{eq.NonLinearObs.CADE}) can be solved using CADE discussed in Section \ref{sec.CADE}, in which one still needs to compute $\nabla \cdot \bp$. To compute $\nabla \cdot \bp$, the value of $\bp$ along $\partial\Omega$ is needed. Here we simply assume zero-Neumann boundary condition, i.e.,
  $$
  \frac{\partial \bp}{\partial \bn}=\mathbf{0}.
  $$
  
  We summarize our algorithms in Algorithm \ref{alg.linear} and \ref{alg.nonlinear}.
  
  \begin{algorithm}[th!]
  	\caption{\label{alg.linear} An operator-splitting method for the linear problem (\ref{eq.obs.linear}).}
  	\begin{algorithmic}
  		\STATE {\bf Input:} The obstacle $\psi$, initial guess $u_0$, external force $f$ and  boundary condition.
  		\STATE {\bf Initialization:} $n=0,$ $u^0=u_0$.
  		\WHILE{not converge}
  		\STATE 1. Solve (\ref{eq.LinearObs.CADE}) for $u^{n+1}$ according to (\ref{eq.cade.1})--(\ref{eq.cade.3}).
  		\STATE 2. Set $n=n+1$.
  		\ENDWHILE
  		\STATE {\bf Output:} The converged function $u^*$.
  	\end{algorithmic}
  \end{algorithm}

\begin{algorithm}[th!]
	\caption{\label{alg.nonlinear} An operator-splitting method for the nonlinear problem (\ref{eq.obs}).}
	\begin{algorithmic}
		\STATE {\bf Input:} The obstacle $\psi$, initial guess $u_0$, external force $f$ and  boundary condition.
		\STATE {\bf Initialization:} $n=0,$ $u^0=u_0$.
		\WHILE{not converge}
		\STATE 1. Solve (\ref{eq.NonLinearObs.op.dis.1}) for $\bp^{n+1/2}$ according to (\ref{eq.NonLinear.p.2}).
		\STATE 2. Solve (\ref{eq.NonLinearObs.op.dis.2}) for $u^{n+1}$ according to (\ref{eq.cade.1})--(\ref{eq.cade.3}).
		\STATE 3. Solve (\ref{eq.NonLinearObs.op.dis.3}) for $\bp^{n+1}$.
		\STATE 4. Set $n=n+1$.
		\ENDWHILE
		\STATE {\bf Output:} The converged function $u^*$.
	\end{algorithmic}
\end{algorithm}

\section{Extensions to other related problems}\label{sec.extension}
In this section, we discuss extensions of the proposed algorithm to other problems related to the obstacle problem. Specifically, we consider the double obstacle problem and the two-phase membrane problem.
\subsection{Double obstacle problem}
\label{sec.double}
The double obstacle problem is similar to the obstacle problem (\ref{eq.obs.linear}) and (\ref{eq.obs}), except now we have two obstacles that bound the solution below and above. Let $\psi$ and $\phi$ be two real-valued functions defined on $\Omega$ such that $\phi\geq \psi$. The double obstacle problem aims to find $u\in \cH^1(\Omega)$ that solves
\begin{align}
	\min_{\substack{u\in \cH^1(\Omega)\\ \psi\leq u \leq \phi}} E_k(u)
	\label{eq.doubleobs}
\end{align}
for $k=1$ or $2$, where $E_1$ and $E_2$ correspond to the linear and nonlinear obstacle problem as defined in (\ref{eq.obs.energy}). The optimality condition of $u$ with $E_1$ is
\begin{align}
	\min\{-\nabla^2 u-f,u-\psi,\phi-u\}=0,
\end{align}
and with $E_2$ is
\begin{align}
	\min\left\{-\nabla\cdot \frac{\nabla u}{\sqrt{|\nabla u|^2+1}} -f,u-\psi, \phi-u\right\}=0.
\end{align}

With a modified CADE method, Algorithm \ref{alg.linear} and \ref{alg.nonlinear} can be applied to solve (\ref{eq.doubleobs}). Specifically, we replace (\ref{eq.cade.2.1}) and (\ref{eq.cade.2.2}) by
\begin{align}
	&u_i^{n+1,1}=\min\left\{\phi,\max\left\{ \psi_i,\zeta\left( u_i^n+\Delta tf_i^n+\frac{\Delta t\eta}{\Delta x^2}\left( u^{n+1,1}_{i-1}-u^n_i+u^n_{i+1}\right)-\frac{\Delta t\eta_2}{2}u^n\right)\right\}\right\},
	\label{eq.cade.double.1}
\end{align}
and
\begin{align}
	u_i^{n+1,2}=\min\left\{\phi,\max\left\{ \psi_i,\zeta\left( u_i^n+\Delta tf_i^n+\frac{\Delta t\eta}{\Delta x^2}\left( u^{n+1,2}_{i+1}-u^n_i+u^n_{i-1}\right)-\frac{\Delta t\eta_2}{2}u^n\right)\right\}\right\},
	\label{eq.cade.double.2}
\end{align}
respectively, where $\zeta$ is defined under (\ref{eq.heat}).

The following theorem shows that if $u^*$ is a solution to the discretized optimality condition of the linear double obstacle problem, then it is a steady state of the modified CADE algorithm using (\ref{eq.cade.double.1}) and (\ref{eq.cade.double.2}).
 \begin{theorem}\label{thm.steady.double}
	For any given function $\psi,\phi$ and $f$ such that $\phi\geq \psi$, let $u^*$ be the solution to the discrete problem
	\begin{align}
		\min\{-\nabla^2_h u-f, u-\psi,\phi-u\}=0,
	\end{align}
	where $\nabla^2_h$ denotes the central difference scheme. Then $u^*$ is a steady state of scheme (\ref{eq.cade.1}), (\ref{eq.cade.double.1}), (\ref{eq.cade.double.2}) and (\ref{eq.cade.3}).
\end{theorem} 
Theorem \ref{thm.steady.double} can be proved analogously to Theorem \ref{thm.steady}, except we need to divided the domain into three regions: the regions on which $u^*=\psi$, $u^*=\phi$ and $-\nabla_h^2 u^*-f=0$. Then the proof follows that of Theorem \ref{thm.steady}.

\subsection{Two-phase membrane problem}
\label{sec.twophase}
The two-phase membrane problem is a free boundary problem and studies the state of an elastic membrane that touches the phase boundary of two phases with different viscosity. When the membrane is pulled away from the phase boundary towards both phases, its equilibrium state solves
\begin{align}
	\min_u \int_{\Omega} \left(\frac{1}{2}|\nabla u|^2  + \mu_1 u_+ -\mu_2 u_-\right) d\bx
	\label{eq.twophase}
\end{align}
where $u_+=\max\{u,0\},\ u_-=\min\{u,0\}$ and $\mu_1,\mu_2>0$ denote forces applied on $u$, see \cite[Section 1.1.5]{petrosyan2012regularity} for details. Following the idea of the proposed algorithm for the nonlinear obstacle problem, we next derive an operator splitting method for (\ref{eq.twophase}). Note that $u_+=(u+|u|)/2$ and $u_-=(u-|u|)/2$. Denote $\lambda_1=(\mu_1-\mu_2)/2,\ \lambda_2=(\mu_1+\mu_2)/2$. Problem (\ref{eq.twophase}) can be rewritten as
\begin{align}
	\min_u \int_{\Omega} \left(\frac{1}{2}|\nabla u|^2  + \lambda_1 u +\lambda_2 |u| \right)d\bx.
	\label{eq.twophase.1}
\end{align}
We decouple the nonlinearity by introducing a variable $v$ and reformulate (\ref{eq.twophase.1}) as
\begin{align} 
	\min_u \int_{\Omega} \left(\frac{1}{2}|\nabla u|^2  + \lambda_1 u +\lambda_2 |v| +\frac{\alpha}{2}|u-v|^2 \right)d\bx ,
\label{eq.twophase.2}
\end{align}
where $\alpha>0$ is a parameter. Denote 
$$
J_4(u)=\int_{\Omega} \left(\frac{1}{2}|\nabla u|^2  + \lambda_1 u\right) d\bx, \ J_5(v)=\int_{\Omega} \lambda_2 |v| d\bx, \ J_6(u,v)=\int_{\Omega} \frac{\alpha}{2}|u-v|^2d\bx.
$$
We associate (\ref{eq.twophase.2}) with the following initial value problem
\begin{align}
	\begin{cases}
		\frac{\partial v}{\partial t} +\partial_v J_5(v)+ D_v J_6(u,v)=0,\\
		\gamma\frac{\partial u}{\partial t}+ D_u J_4(u)+D_u J_6(u,v)=0,		
	\end{cases}
\label{eq.twophase.3}
\end{align}
which is then time--discretized as
\begin{align}
	\begin{cases}
		\frac{v^{n+1}-v^n}{\Delta t} +\partial_v J_5(v^{n+1})+ D_v J_6(u^n,v^{n+1})=0,\\
		\gamma\frac{u^{n+1}-u^n}{\Delta t}+ D_u J_4(\widetilde{u}^{n+1/2})+D_u J_6(\widetilde{u}^{n+1/2},v^{n+1})=0,	
	\end{cases}
	\label{eq.twophase.4}
\end{align}
where $D_u J_4(\widetilde{u}^{n+1/2})$ denotes the ADE treatment of $D_u J_4$ using $u^{n+1}$ and $u^n$, and $D_u J_6(\widetilde{u}^{n+1/2},v^{n+1})$ is the ADE treatment of $D_u J_6$ using $u^{n+1},u^n$ and $v^{n+1}$, see (\ref{eq.ade.2.1}) and (\ref{eq.ade.2.2}) for details.

For the updating formula of $v^{n+1}$, note that it solves 
\begin{align}
	v^{n+1}=\argmin_v \int_{\Omega} \left( \frac{1}{2} |v-v^n|^2+ \lambda_2\Delta t |v| + \frac{\alpha\Delta t}{2}|v-u^n|^2 \right)d\bx.
\end{align}
We have the explicit formula using the shrinkage operator
\begin{align}
	v^{n+1}=\max\left\{ 0,\frac{1-\lambda_2\Delta t}{|v^n+\alpha\Delta t u^n|}\right\} \frac{v^n+\alpha\Delta t u^n}{1+\alpha\Delta t}.
\end{align}

To update $u^{n+1}$, since there is no constraints for $u$, we will use the ADE method:
\begin{align}
	u^{n+1}=\ADE(\nabla^2-\alpha I,u^{n+1},u^n,u^n+\alpha v^{n+1},\Delta t/\gamma).
\end{align}


  \section{Numerical experiments}\label{sec.experiments}
In this section, we demonstrate the performance and efficiency of the proposed methods. All experiments are conducted using MATLAB R2020(b) on a Windows desktop of 16GB RAM and Intel(R) Core(TM) i7-10700 CPU: 2.90GHz. In all examples, we use Cartesian grids either on an interval for one-dimensional problems or on a rectangular domain for two-dimensional problems. Without specification, we set $f=0$, i.e., there is no external force. For both algorithms, we use the stopping criterion 
  $$
  \|u^{n+1}-u^n\|_{\infty}< tol
  $$
  for a small $tol>0$. When the exact solution $u^*$ is given, for one-dimensional problems, we define the $L^2$ error and $L^{\infty}$ error of $u$ as
  \begin{align}
  	\|u-u^*\|_2=\sqrt{\sum_i (u_i-u_i^*)^2\Delta x} \mbox{ and } \|u-u^*\|_{\infty}=\max_i |u_i-u^*_i|,
  \end{align}
respectively.
For two dimensional problems, we use $\Delta x=\Delta y$ and define
\begin{align}
	\|u-u^*\|_2=\sqrt{\sum_{i,j}(u_{i,j}-u_{i,j}^*)^2\Delta x^2} \mbox{ and } \|u-u^*\|_{\infty}=\max_{i,j} |u_{i,j}-u^*_{i,j}|.
\end{align}

  \subsection{One dimensional examples of the linear obstacle problem}
  \label{sec.numerical.linear.1D}
  We first test Algorithm \ref{alg.linear} on three one-dimensional examples of the linear problems considered in \cite{tran20151,zosso2017efficient}. Consider the obstacles defined as
  \begin{align}
  	\psi_1(x)=\begin{cases}
  		100x^2 & \mbox{ for } 0\leq x\leq 0.25,\\
  		100x(1-x)-12.5 &\mbox{ for } 0.25\leq x\leq 0.5,\\
  		\psi_1(1-x) & \mbox{ for } 0.5\leq x\leq 1,
  	\end{cases}
  \label{eq.ex.obs1.1d.1}\\
  	\psi_2(x)= 
  \begin{cases}
  	10\sin(2\pi x) & \mbox{ for } 0\leq x\leq 0.25,\\
  	5\cos(\pi(4x-1))+5 & \mbox{ for } 0.25\leq x\leq 0.5,\\
  	\psi_2(1-x) & \mbox{ for } 0.5\leq x \leq 1,
  \end{cases}
  \label{eq.ex.obs1.1d.2}
  \end{align}
and
\begin{align}
	\psi_3(x)= 10\sin^2(\pi(x+1)^2), \quad 0\leq x \leq 1.
\label{eq.ex.obs1.1d.3}
\end{align}
For both examples, our computational domain is $[0,1]$. We use boundary condition $u(0)=u(1)=0$ for $\psi_1$ and $\psi_2$, and $u(0)=5,\ u(1)=10$ for $\psi_3$. In Algorithm \ref{alg.linear}, we set $\Delta t=0.1\Delta x$. With $\Delta x=1/256$, the obstacles and our numerical solutions are shown in the first row of Figure \ref{fig.linear1d.general}. For all examples, except for the region that the solution contacts the obstacle, the solution is linear. The histories of the error $\|u^{n+1}-u^n\|_{\infty}$ for the three experiments are shown in the second row. Linear convergence is observed.

\begin{figure}[ht]
	\centering
	\begin{tabular}{ccc}
		(a) & (b) & (c)\\
		\includegraphics[width=0.3\textwidth]{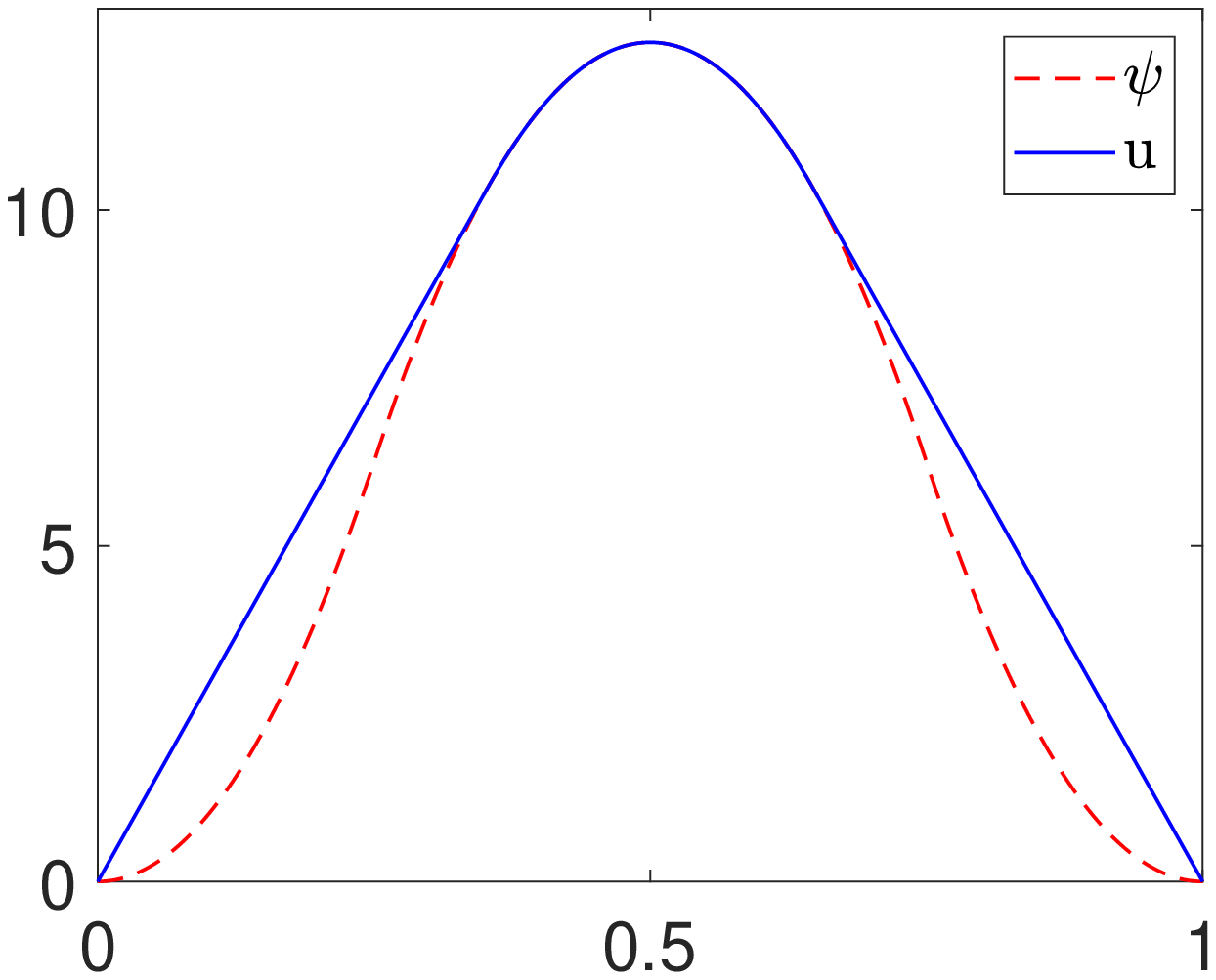} & 
		\includegraphics[width=0.3\textwidth]{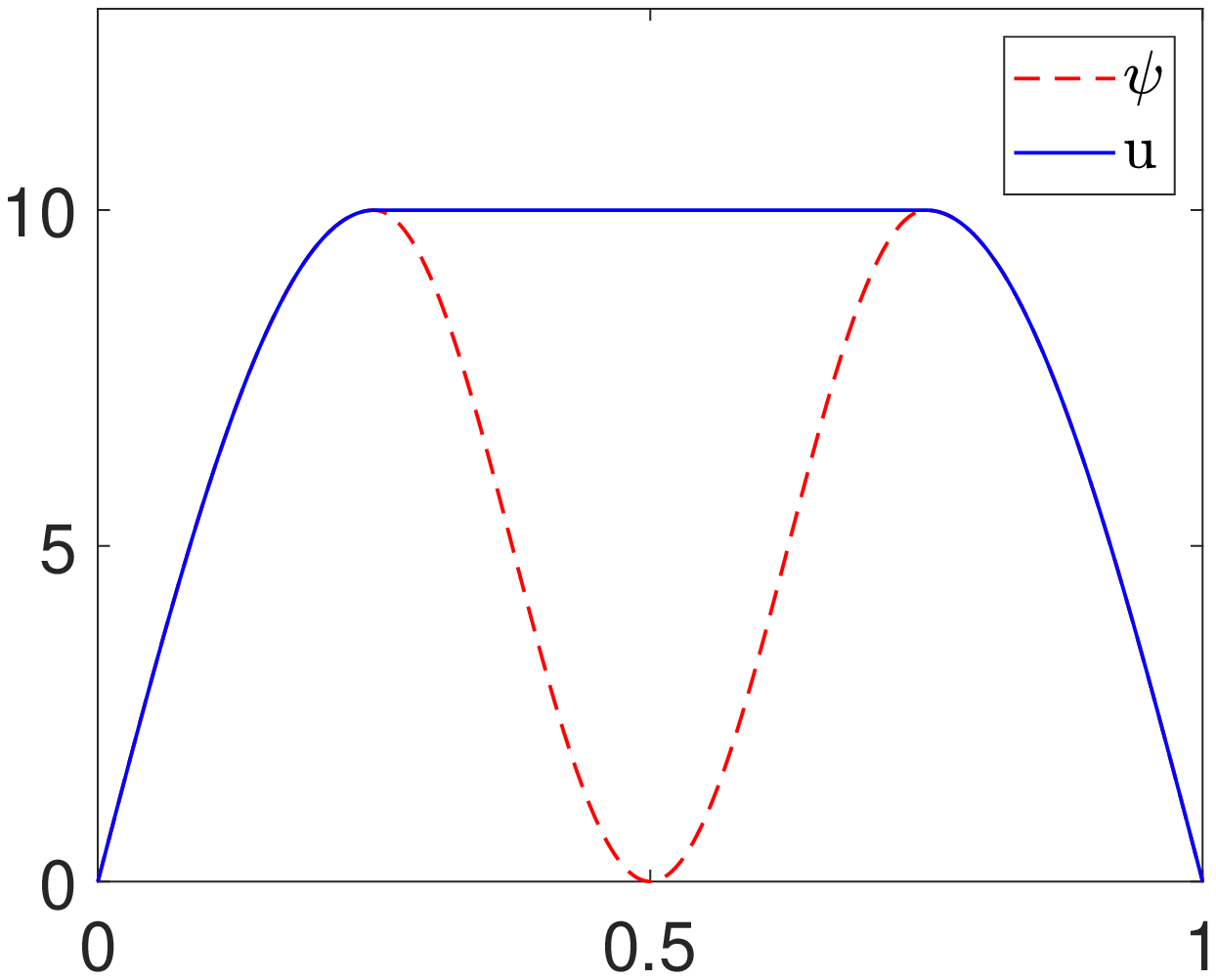} & 
		\includegraphics[width=0.3\textwidth]{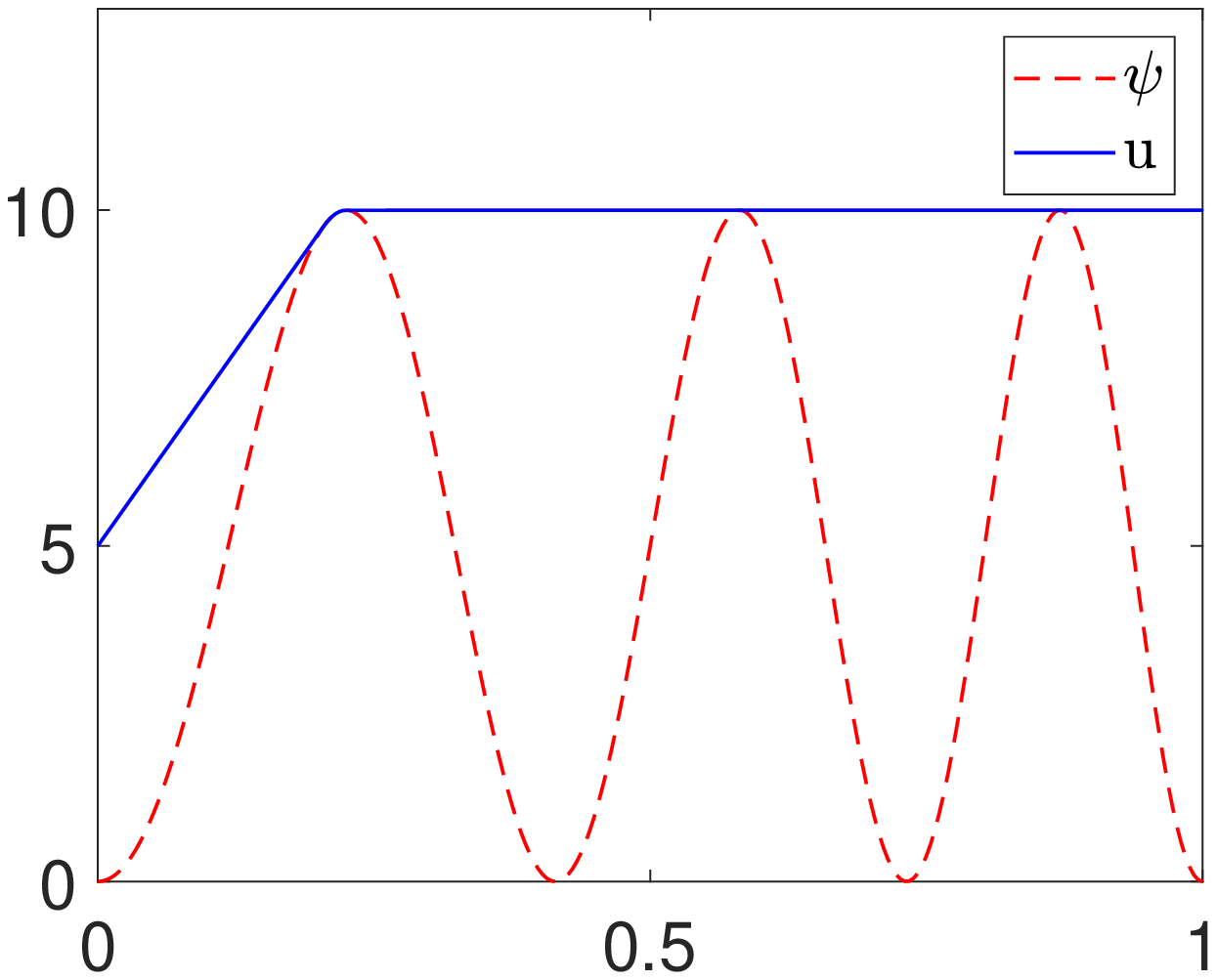}\\
		\includegraphics[width=0.3\textwidth]{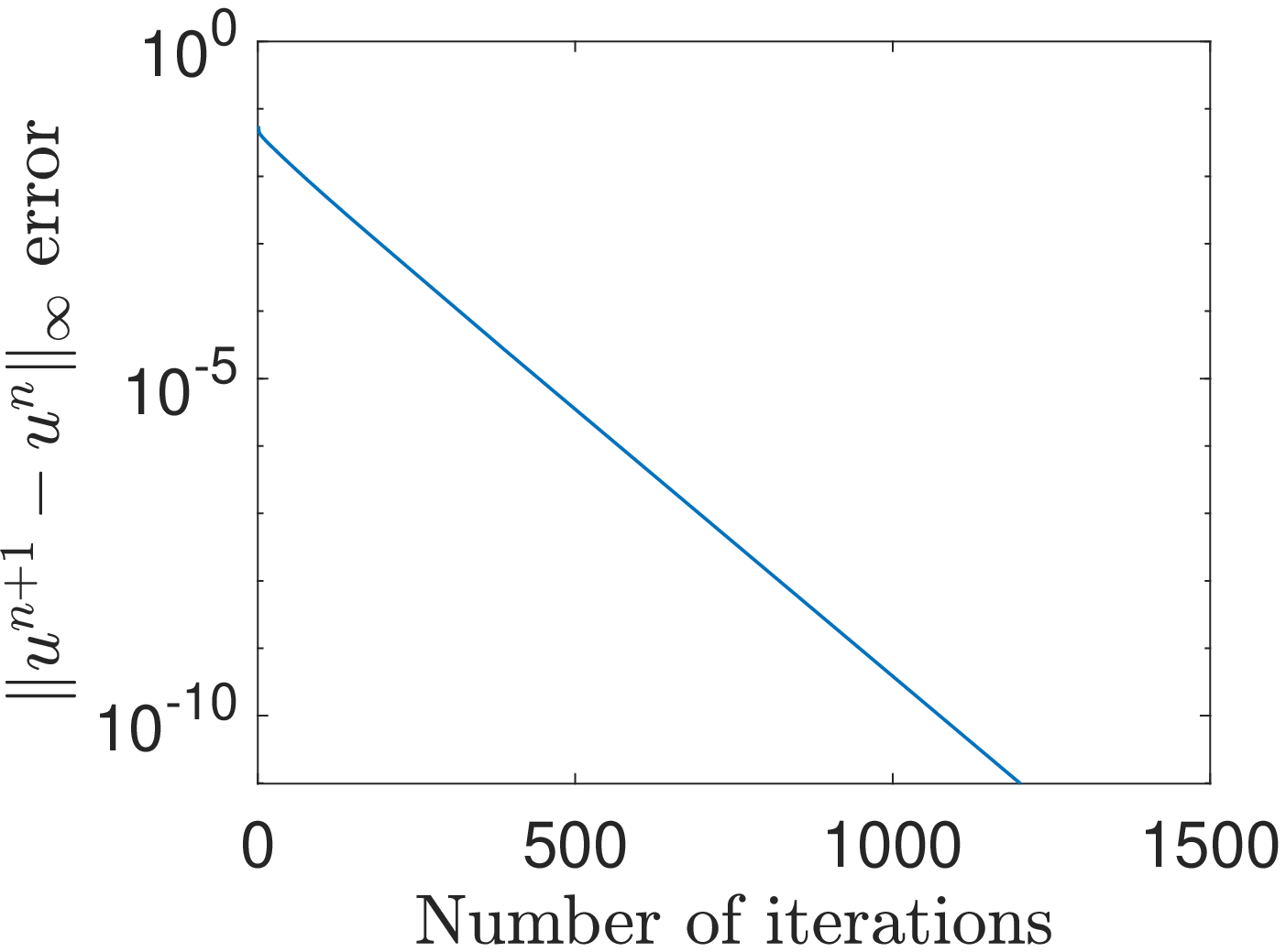} &
		\includegraphics[width=0.3\textwidth]{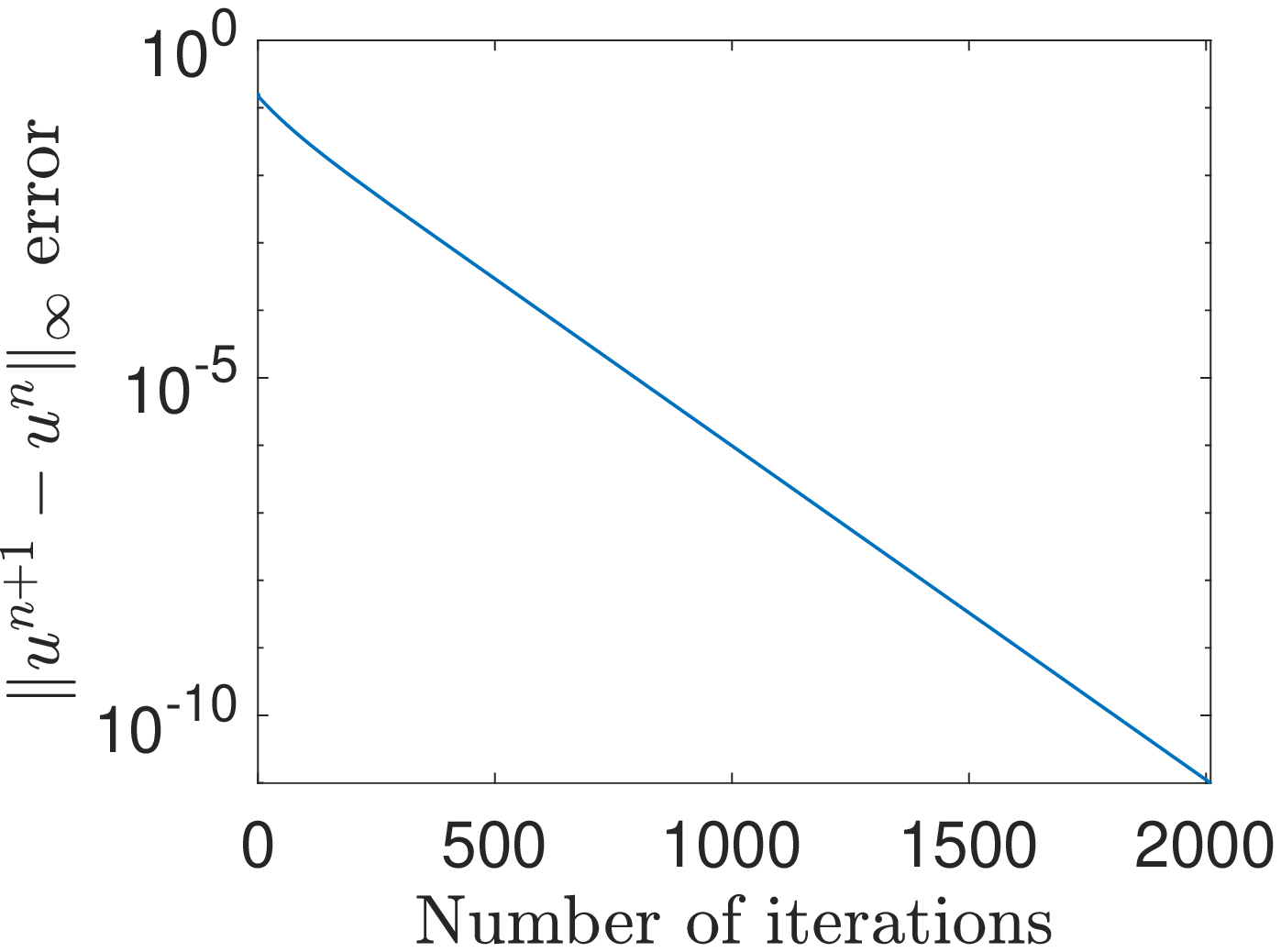} &
		\includegraphics[width=0.3\textwidth]{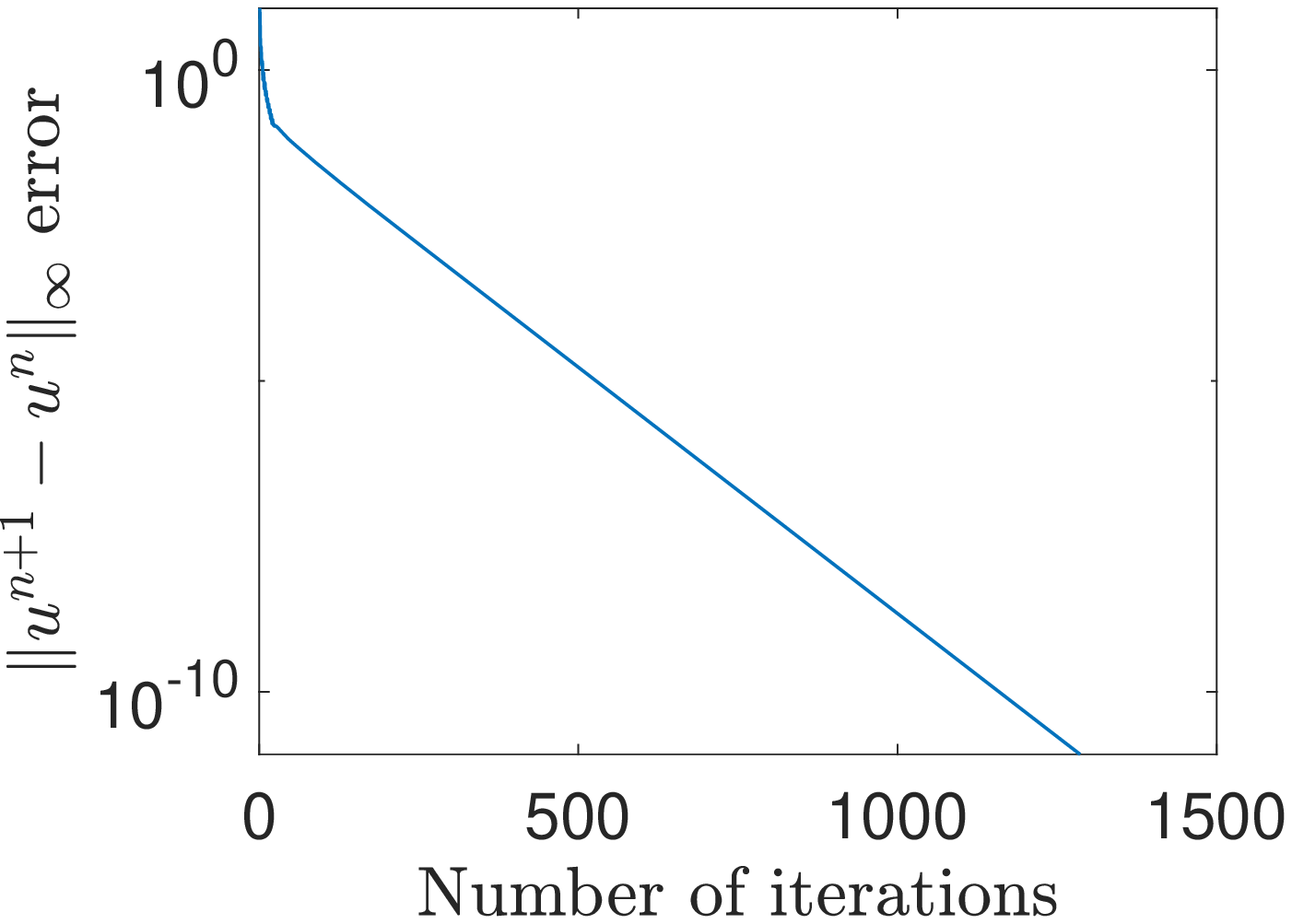} 
	\end{tabular}
	\caption{(One dimensional linear problems with $\Delta x=1/256$.) Column (a)-(c) correspond to obstacles $\psi_1,\ \psi_2,\ \psi_3$, respectively. First row: Graph of the obstacle and the numerical solution by Algorithm \ref{alg.linear}. Second row: Histories of the error $\|u^{n+1}-u^n\|_{\infty}$ .}
	\label{fig.linear1d.general}
\end{figure}

For $\psi_1$, the exact solution is given as
\begin{align}
	u_1^*(x)=\begin{cases}
		(100-50\sqrt{2})x & \mbox{ for } 0\leq x\leq \frac{1}{2\sqrt{2}},\\
		100x(1-x)-12.5 & \mbox{ for } \frac{1}{2\sqrt{2}}\leq x \leq 0.5,\\
		u_1^*(1-x) & \mbox{ for } 0.5\leq x \leq 1,
	\end{cases}
\end{align}
We present the evolution of the $L^2$ and $L^{\infty}$ error of our solution in Figure \ref{fig.linear1d.err}. Again, linear convergence is observed. Both errors achieve their minimum with around 600 iterations.

\begin{figure}[ht]
	\centering
	\begin{tabular}{ccc}
		\includegraphics[width=0.4\textwidth]{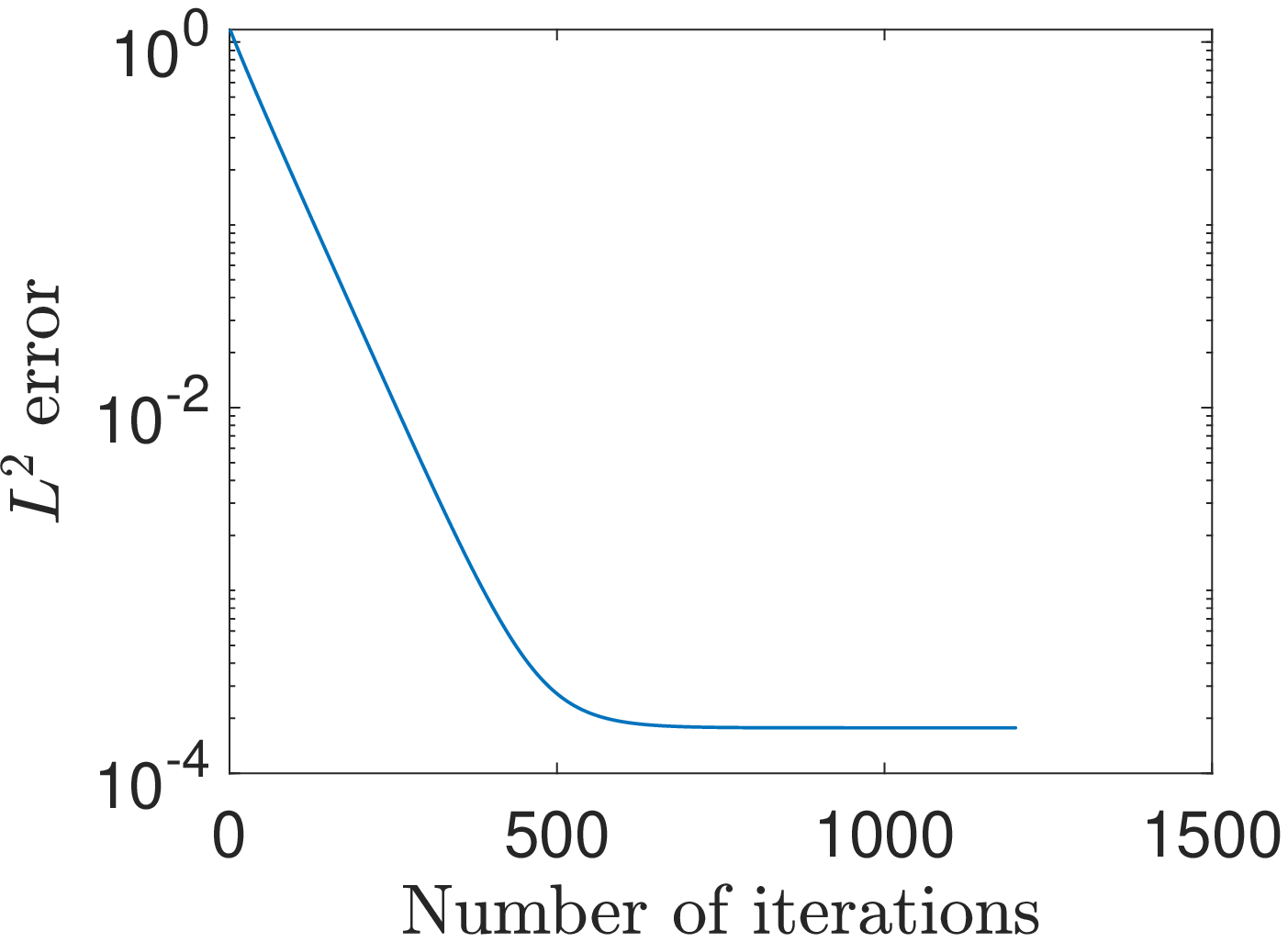} &
		\includegraphics[width=0.4\textwidth]{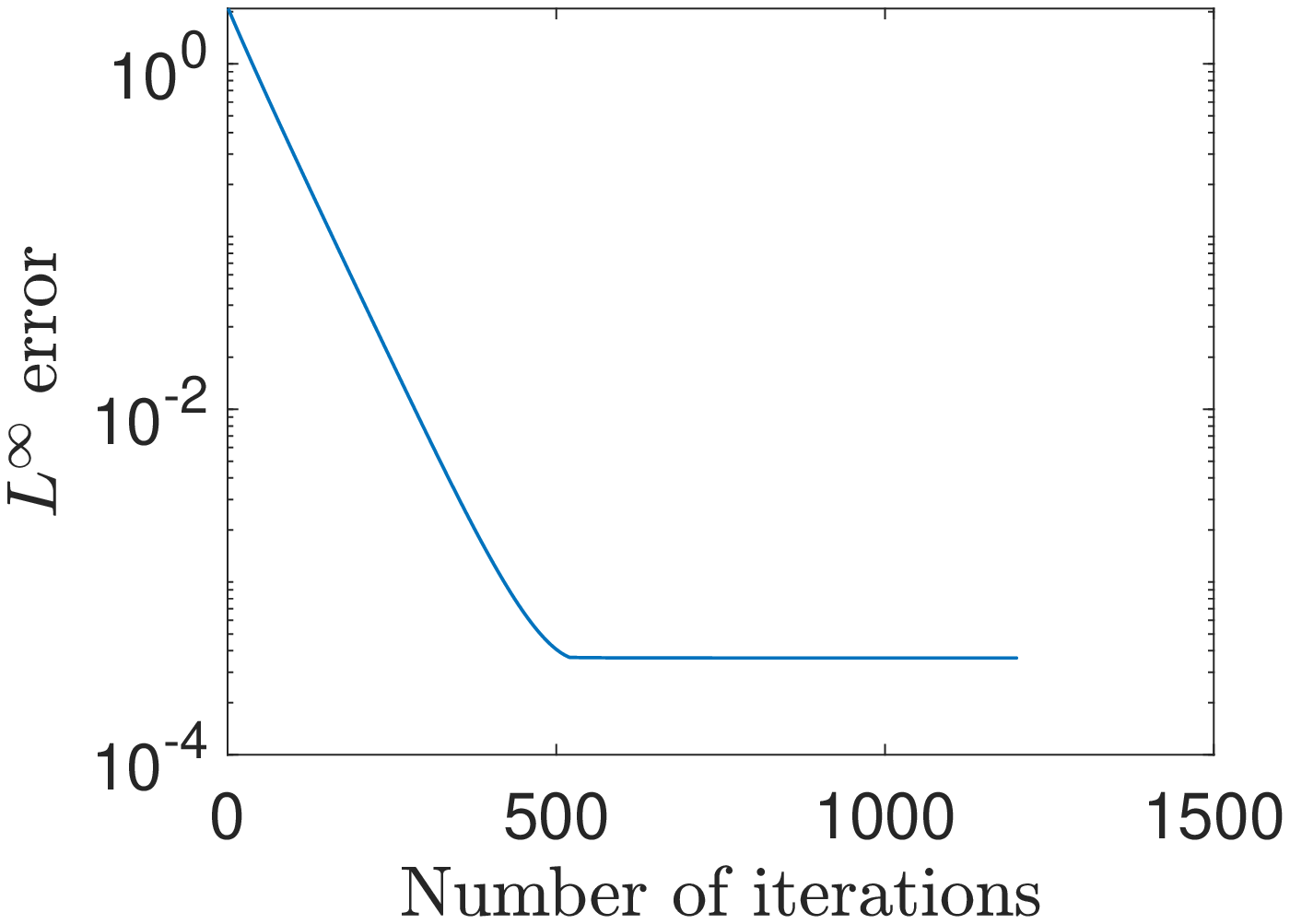} 
	\end{tabular}
	\caption{(One dimensional linear problem with obstacle $\psi_1$ and $\Delta x=1/256$.) History of (a) the $L^2$ error, and (b) the $L^{\infty}$ error of our numerical solution by Algorithm \ref{alg.linear}.}
	\label{fig.linear1d.err}
\end{figure}

We then compare our Algorithm \ref{alg.linear} with the primal-dual (PD) method in \cite{zosso2017efficient} and the split Bregman (SB) method in \cite{tran20151} on obstacle $\psi_1$. In each iteration of SB, one needs to solve a linear system, for which we use MATLAB's backslash which determines the best solver by itself. We set the stopping criterion as $\|u^n-u^{n+1}\|_{\infty}\leq 10^{-11}$ for all methods. The results by all methods have similar $L^2$ error and $L^{\infty}$ error. We compare their computational costs in Table \ref{tab.linear1d}. Our proposed method is the most efficient one.

\begin{table}[ht]
	\centering
	\begin{tabular}{c||rr|rr|rr}
		\hline
		$\Delta x$ & \multicolumn{2}{c|}{Our algorithm} & \multicolumn{2}{c|}{PD} &  \multicolumn{2}{c}{SB} \\
		\hline
		1/64 &  299 &(0.0004) & 2488 &(0.0162) & 3273 &(0.1223) \\
		\hline
		1/128 & 595 &(0.0014) & 3191 &(0.0257) & 13194 &(1.5670) \\
		\hline
		1/256 & 1201 &(0.0068) & 3796 &(0.0362) & 36166 &(14.5045)  \\
		\hline
		1/512 & 2363 &(0.0163) & 3955& (0.0533) & 92293 &(200.9334) \\
		\hline
	\end{tabular}
	\caption{\label{tab.linear1d}(One dimensional linear problem with obstacle $\psi_1$.) Comparison of the number of iterations (CPU time in seconds) required to satisfy the stopping criterion of the propose algorithm, the primal-dual method in \cite{zosso2017efficient}, and the split Bregman method in \cite{tran20151}.}
	
\end{table}

\begin{figure}[t!]
	\centering
	\begin{tabular}{cc}
		(a) & (b) \\
		\includegraphics[width=0.4\textwidth]{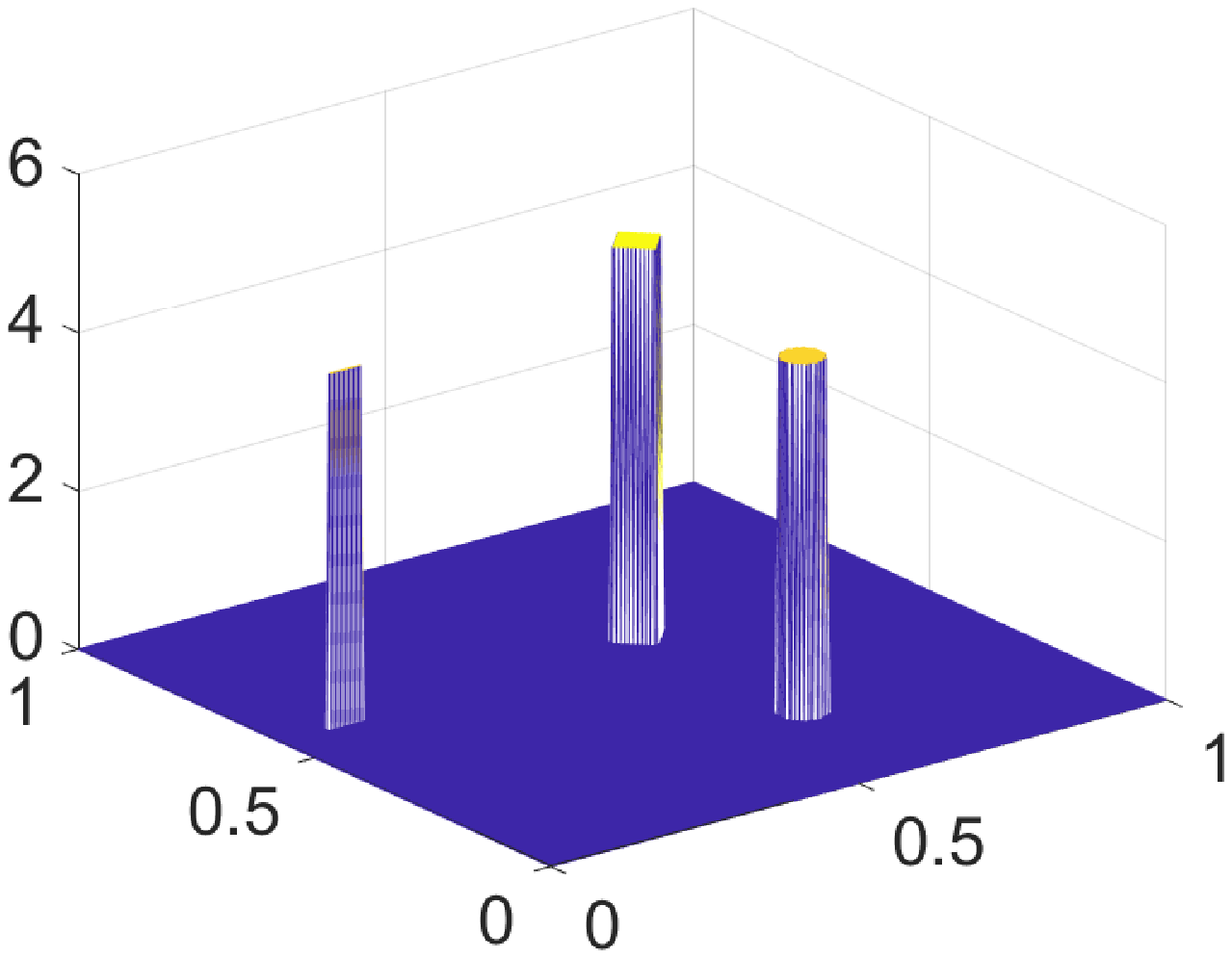} & 
		\includegraphics[width=0.4\textwidth]{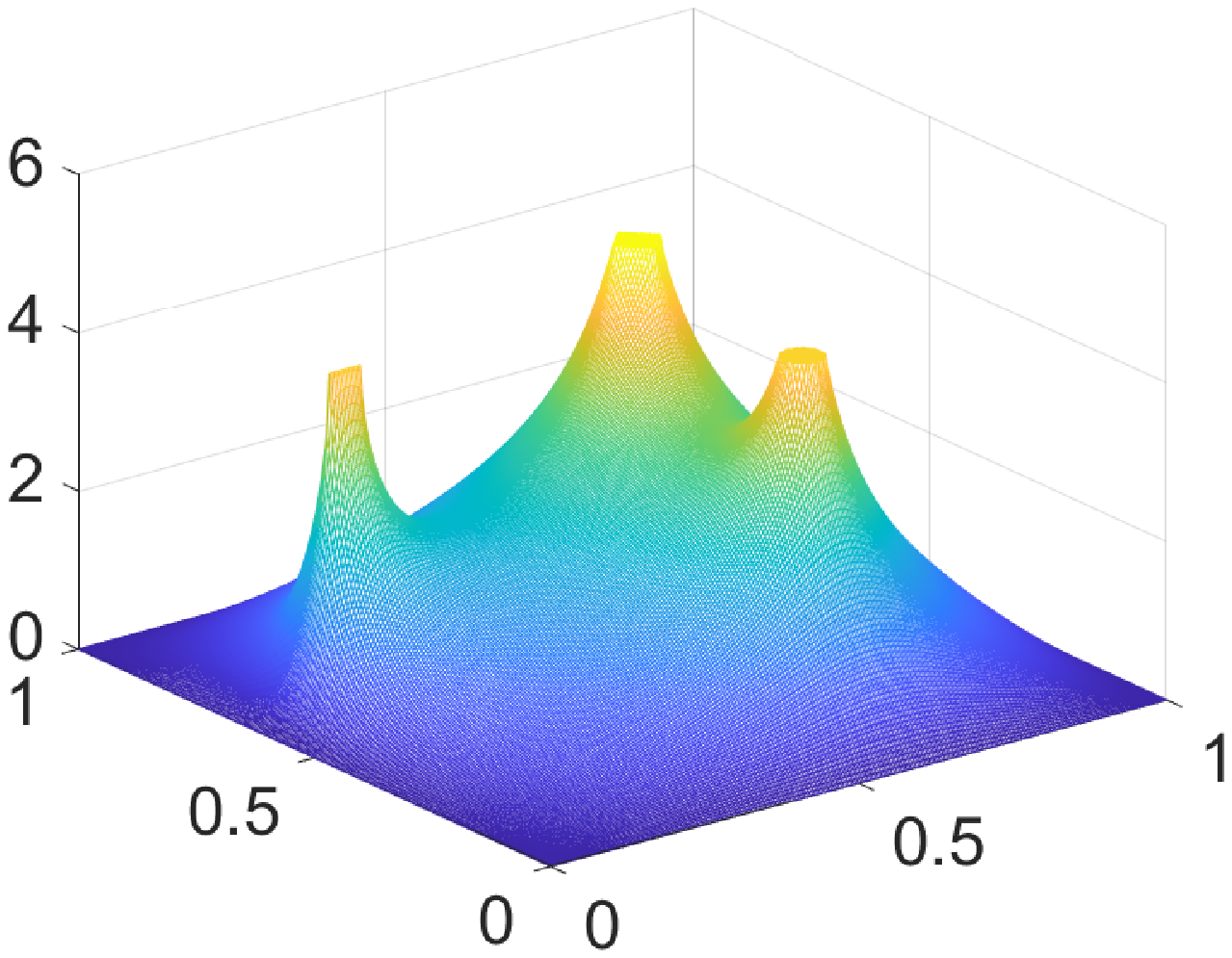}\\
		(c) & (d) \\
		\includegraphics[width=0.4\textwidth]{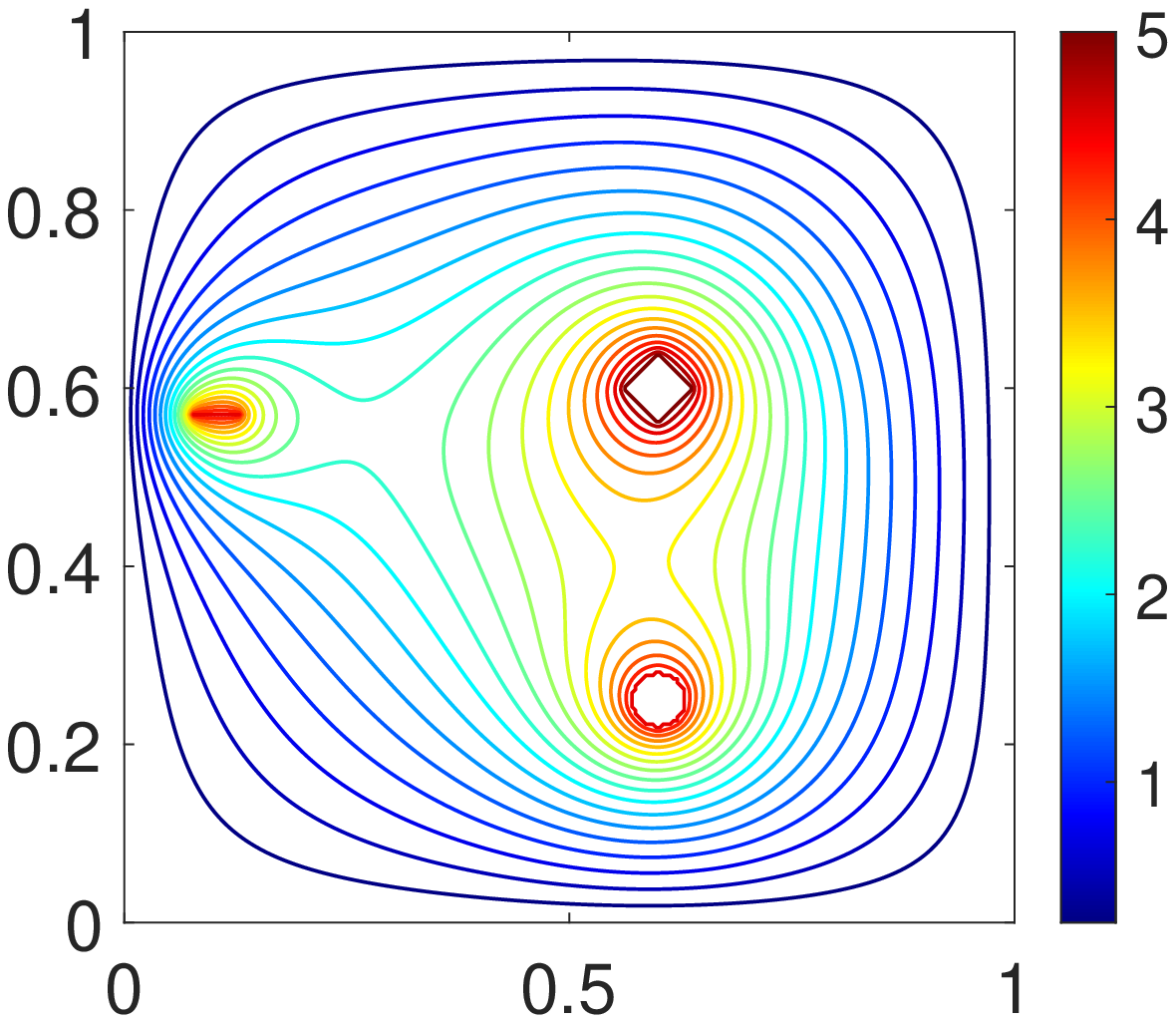}&
		\includegraphics[width=0.4\textwidth]{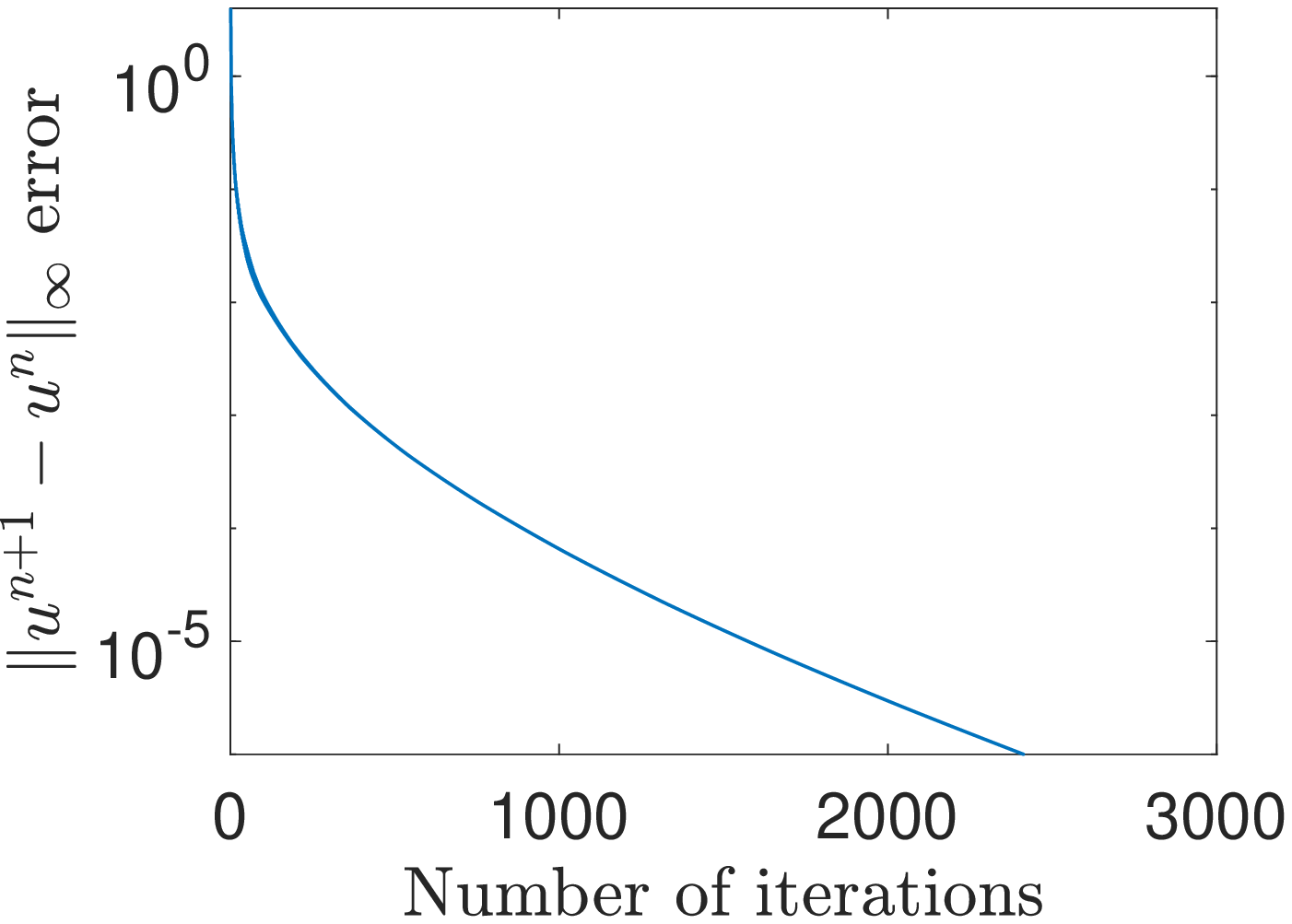}
	\end{tabular}
	\caption{(Two dimensional linear problem with obstacle $\psi_4$ and $\Delta x=1/256$.) (a) Graph of the obstacle $\psi_4$. (b) Graph of the numerical solution by Algorithm \ref{alg.linear}. (c) Contour of the numerical solution. (d) History of the error $\|u^{n+1}-u^n\|_{\infty}$.}
	\label{fig.linear.psi4}
\end{figure}

\begin{figure}[t!]
	\begin{tabular}{cc}
		(a) &(b)\\
		\includegraphics[width=0.4\textwidth]{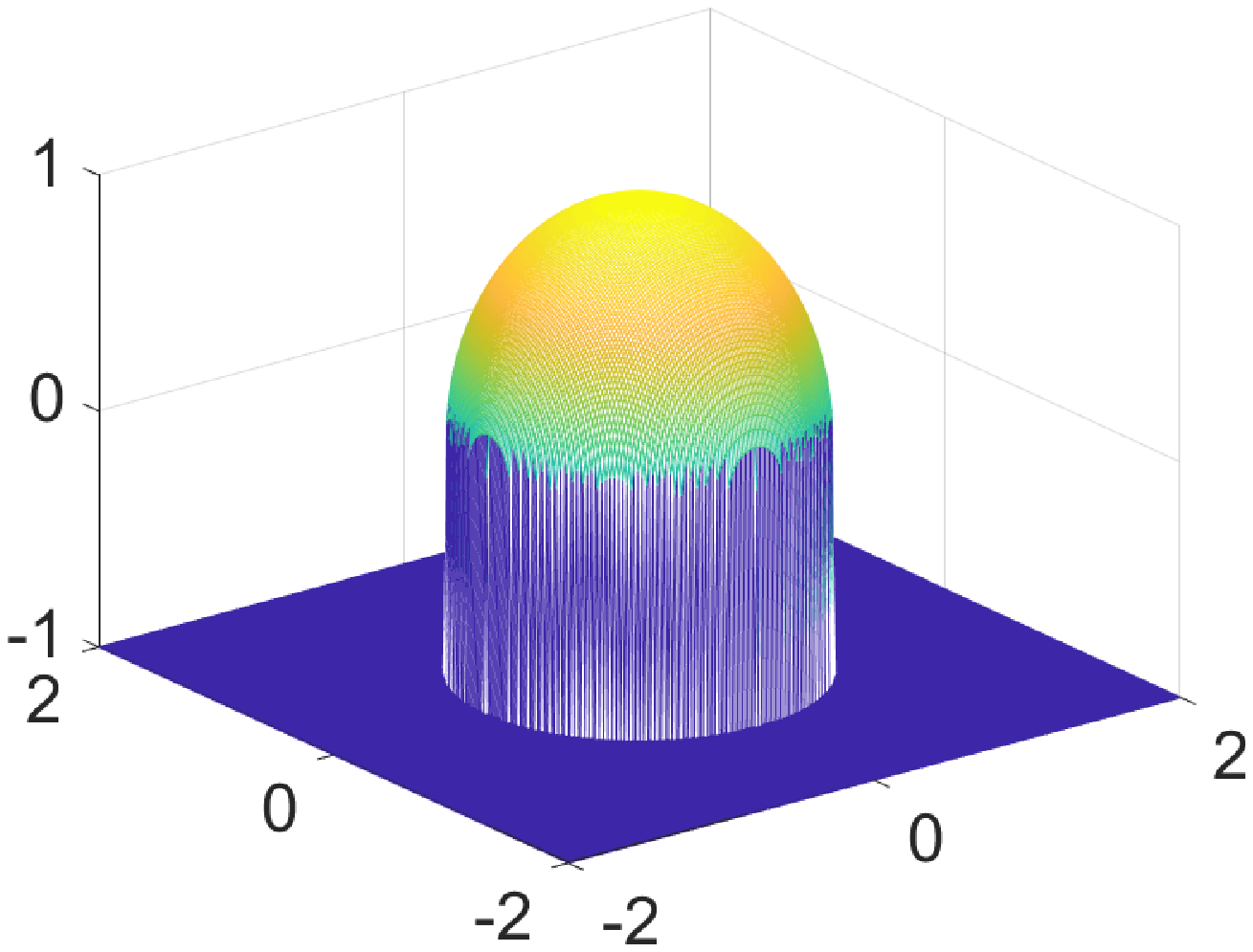} & 
		\includegraphics[width=0.4\textwidth]{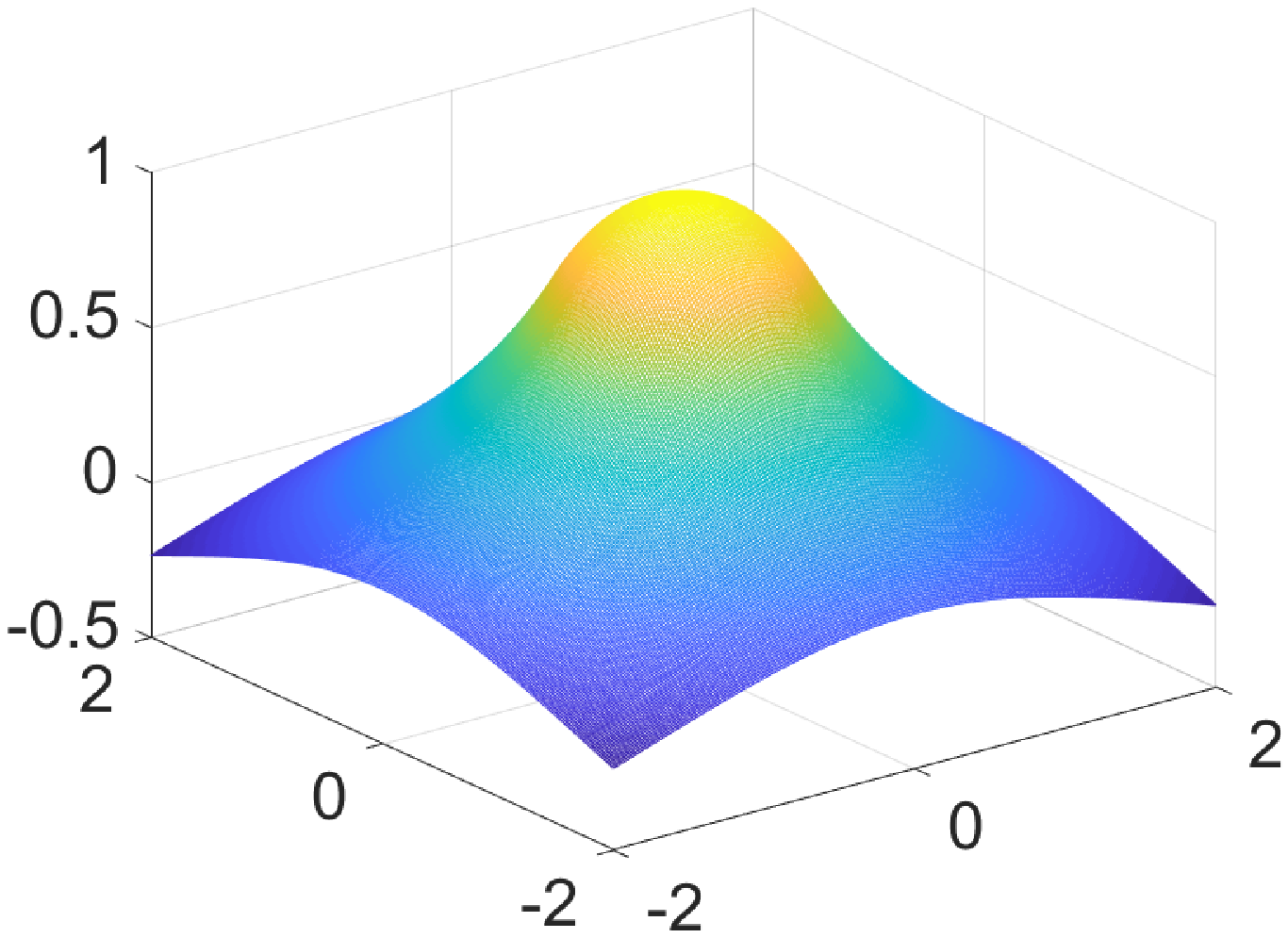}\\
		(c) & (d)\\
		\includegraphics[width=0.4\textwidth]{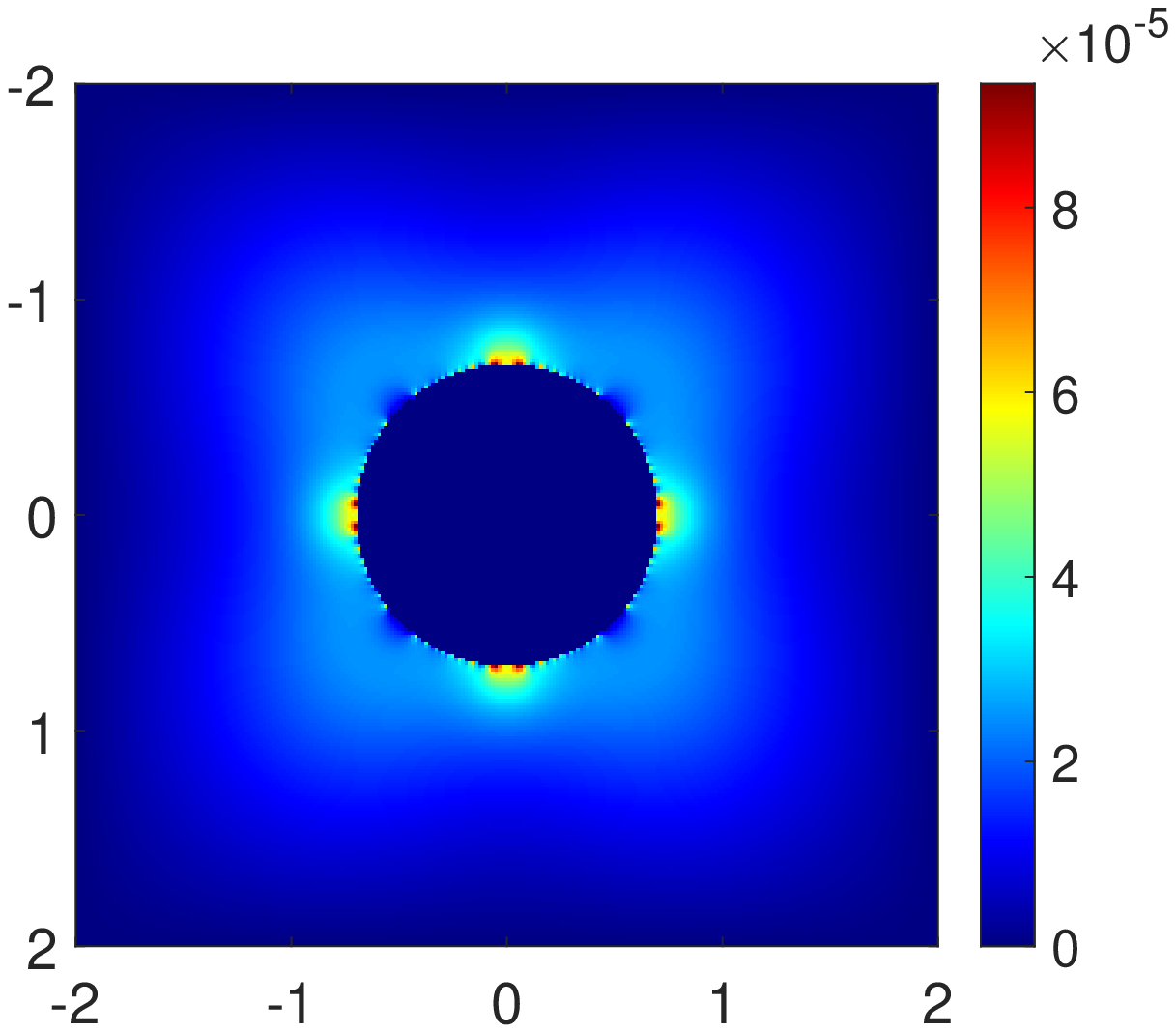} & 
		\includegraphics[width=0.4\textwidth]{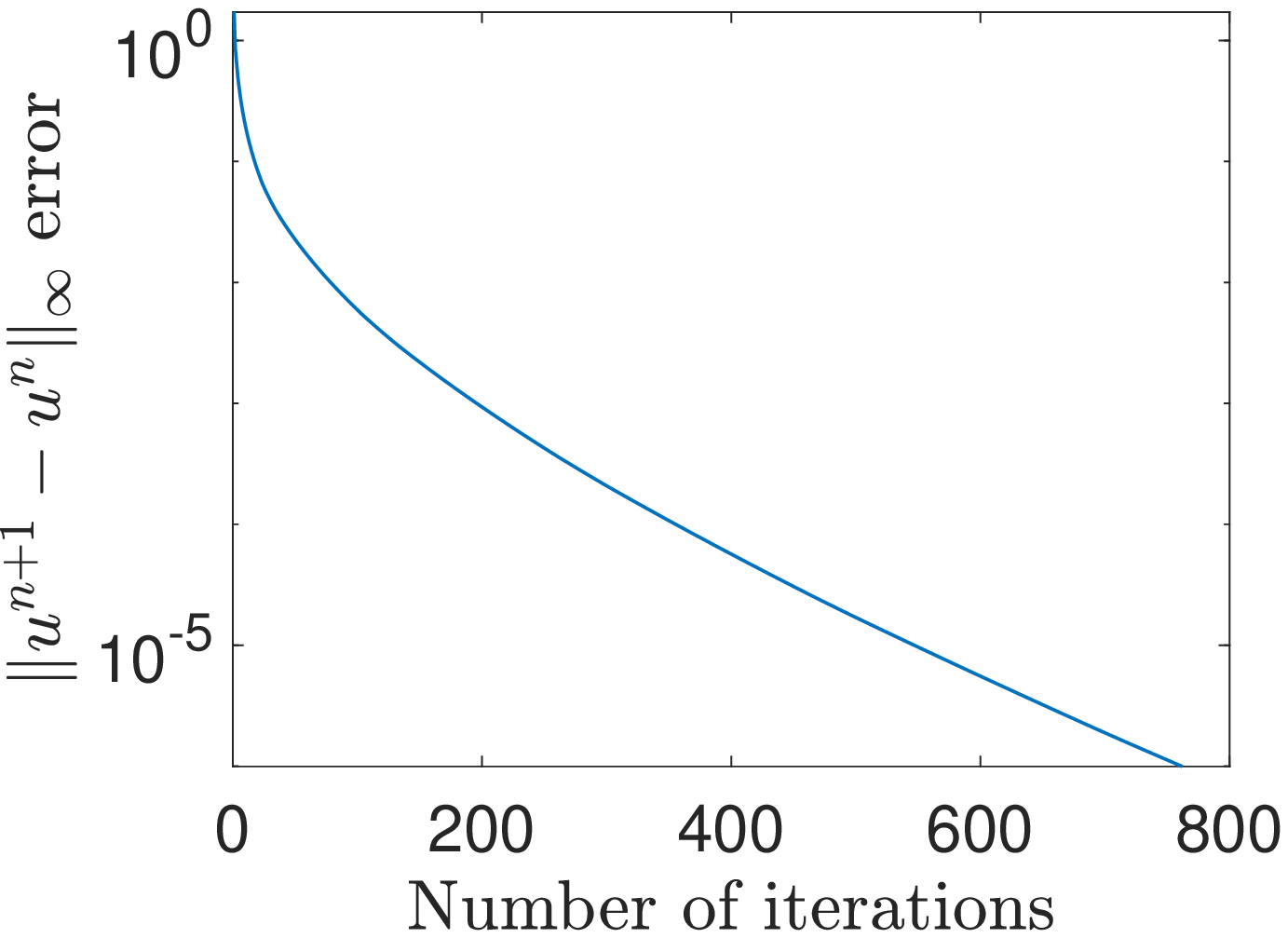}
	\end{tabular}
	\caption{(Two dimensional linear problem with $\Delta x=\Delta y=1/256$.) (a) Graph of the obstacle $\psi_5$. (b) Graph of the numerical solution by Algorithm \ref{alg.linear}. (c) Color map of the error $|u^n-u^*|$. (d) The history of the error $\|u^{n+1}-u^n\|_{\infty}$.}
	\label{fig.linear2d.psi5}
\end{figure}

\begin{table}[t!]
	\centering
	(a) \\
	\begin{tabular}{c||c|c|c}
		\hline
		$\Delta x$& Algorithm \ref{alg.linear} & PD & SB\\
		\hline
		1/32 & $4.94\times 10^{-3} $  & $4.92\times 10^{-3} $  & $4.94\times 10^{-3} $ \\
		\hline
		1/64 &  $5.85\times 10^{-4}$  & $5.55\times 10^{-4}$  &  $5.85\times 10^{-4}$  \\
		\hline
		1/128 & $1.89\times 10^{-4}$  & $1.93\times 10^{-4}$  & --\\
		\hline
		1/256 &  $5.23\times 10^{-5}$  & $5.33\times 10^{-5}$ & -- \\
		\hline
		Order &  2.19&  2.17  & -- \\
		\hline
	\end{tabular}
	\vspace{0.2cm}
	
	(b) \\
	\begin{tabular}{c||c|c|c}
		\hline
		$\Delta x$& Algorithm \ref{alg.linear} & PD & SB\\
		\hline
		1/32 &   $5.75\times 10^{-3}$ & $5.74\times 10^{-3}$ & $5.75\times 10^{-3}$\\
		\hline
		1/64 &   $5.99\times 10^{-4}$ &  $5.98\times 10^{-4}$ &   $5.99\times 10^{-4}$ \\
		\hline
		1/128 &  $2.15\times 10^{-4}$ &  $2.16\times 10^{-4}$ & --\\
		\hline
		1/256 &   $9.34\times 10^{-5}$ &  $9.34\times 10^{-5}$ &  --\\
		\hline
		Order &   1.98 &  1.98 &  --\\
		\hline
	\end{tabular}
	
	\caption{\label{tab.linear2d.err} (Two dimensional linear problem with obstacle $\psi_5$.)
		Comparison of (a) the $L^2$ error and (b) the $L^{\infty}$ error of the propose algorithm, the primal-dual method in \cite{zosso2017efficient} and the split Bregman method in \cite{tran20151}. Second order convergence is observed. }
\end{table}

\begin{table}[t!]
	\centering
	\begin{tabular}{c||cc|cc|cc}
		\hline
		$\Delta x$ & \multicolumn{2}{c|}{Algorithm \ref{alg.linear}} & \multicolumn{2}{c|}{PD}  &  \multicolumn{2}{c}{SB}\\
		\hline
		1/32 &  209 &(0.0082)  & 1401 &(0.0391) & 203 &(1.6712)\\
		\hline
		1/64 &  405 &(0.0801)  & 2124 &(0.2380) & 532 &(112.4426)\\
		\hline
		1/128 & 776 &(0.5667)   & 3054 &(1.1650) & \multicolumn{2}{c}{--}\\
		\hline
		1/256 & 1484 &(4.1424)   & 3436 &(8.8987) & \multicolumn{2}{c}{--}\\
		\hline
	\end{tabular}
	\caption{\label{tab.linear2d.time} (Two dimensional linear problem with obstacle $\psi_5$.) Comparison of the number of iterations (CPU time in seconds) required to satisfy the stopping criterion of the propose algorithm, the primal-dual method in \cite{zosso2017efficient} and the split Bregman method in \cite{tran20151}.}
\end{table}

\subsection{Two dimensional examples of the linear obstacle problem} 
We test Algorithm \ref{alg.linear} on two dimensional obstacles. The first experiment considers computational domain $[0,1]^2$ and the obstacle with disjoint bumps 
\begin{align}
	\psi_4(x,y)=\begin{cases}
		5 & \mbox{ for } |x-0.5|+|y-0.6|<0.04,\\
		4.5 &\mbox{ for } (x-0.6)^2+(y-0.25)^2<0.001,\\
		4.5 & \mbox{ for } y=0.57 \mbox{ and } 0.075<x<0.13,\\
		0 &\mbox{ otherwise}.
	\end{cases}
\label{eq.obs4}
\end{align}
The obstacle is visualized in Figure \ref{fig.linear.psi4}(a). With $\Delta x=1/256$ and zero boundary condition, our numerical solution by Algorithm \ref{alg.linear} and its contour is shown in (b) and (c), respectively. The numerical solution is smooth away from the obstacle and contacts the obstacle over the support of the obstacle. We present the history of the error $\|u^{n+1}-u^n\|_{\infty}$ in (d).

We then test Algorithm \ref{alg.linear} on an example with known exact solution.
 Consider the computational domain $[-2,2]^2$. We define the obstacle as
\begin{align}
	\psi_5(x,y)= \begin{cases}
		\sqrt{1-x^2-y^2} & \mbox{ for } \sqrt{x^2+y^2}\leq 1,\\
		-1 & \mbox{ otherwise}.
	\end{cases}
\end{align}
The exact solution is given as
\begin{align}
	u^*(x,y)=\begin{cases}
		\sqrt{1-x^2-y^2} & \mbox{ for } \sqrt{x^2+y^2}\leq r^* ,\\
		-(r^*)^2\log(r/2)/\sqrt{1-(r^*)^2} & \mbox{ for } \sqrt{x^2+y^2}\geq r^*,
	\end{cases}
\end{align}
where $r^*$ is the solution of
$$
(r^*)^2(1-\log(r^*/2))=1.
$$
One has $r^*=0.6979651482...$ \cite{zosso2017efficient}. In this experiment, we set the boundary condition as the boundary value of $u^*$, and $\Delta t=\Delta x$. With $\Delta x=1/256$, the graph of the obstacle and our numerical solution are shown in Figure \ref{fig.linear2d.psi5} (a) and (b), respectively. In (c), we present the color map of the error $|u^n-u^*|$, from which we observe that the error is in the order of $10^{-5}$ and concentrates around the contacting region. The history of the error $\|u^{n+1}-u^n\|_{\infty}$ is presented in (d).

We then compare Algorithm \ref{alg.linear} with PD and SB on this example. For Algorithm \ref{alg.linear} and PD, we test on grid resolution from $32\times 32$ to $256\times 256$. For SB, since solving a linear system in each iteration is very time consuming on fine grids, we test it on grid resolution up to $64\times 64$. In Table \ref{tab.linear2d.err}, we compare the $L^2$ errors and $L^{\infty}$ errors. All methods provides results with similar errors. Second order convergence is observed. We then compare the efficiency of these methods in Table \ref{tab.linear2d.time}, in which the number of iterations and CPU time required for convergence are reported. The same to our observation in the previous subsection, Algorithm \ref{alg.linear} is the most efficient.

\begin{figure}[t!]
	\centering
	\begin{tabular}{ccc}
		(a) & (b) & (c)\\
		\includegraphics[width=0.3\textwidth]{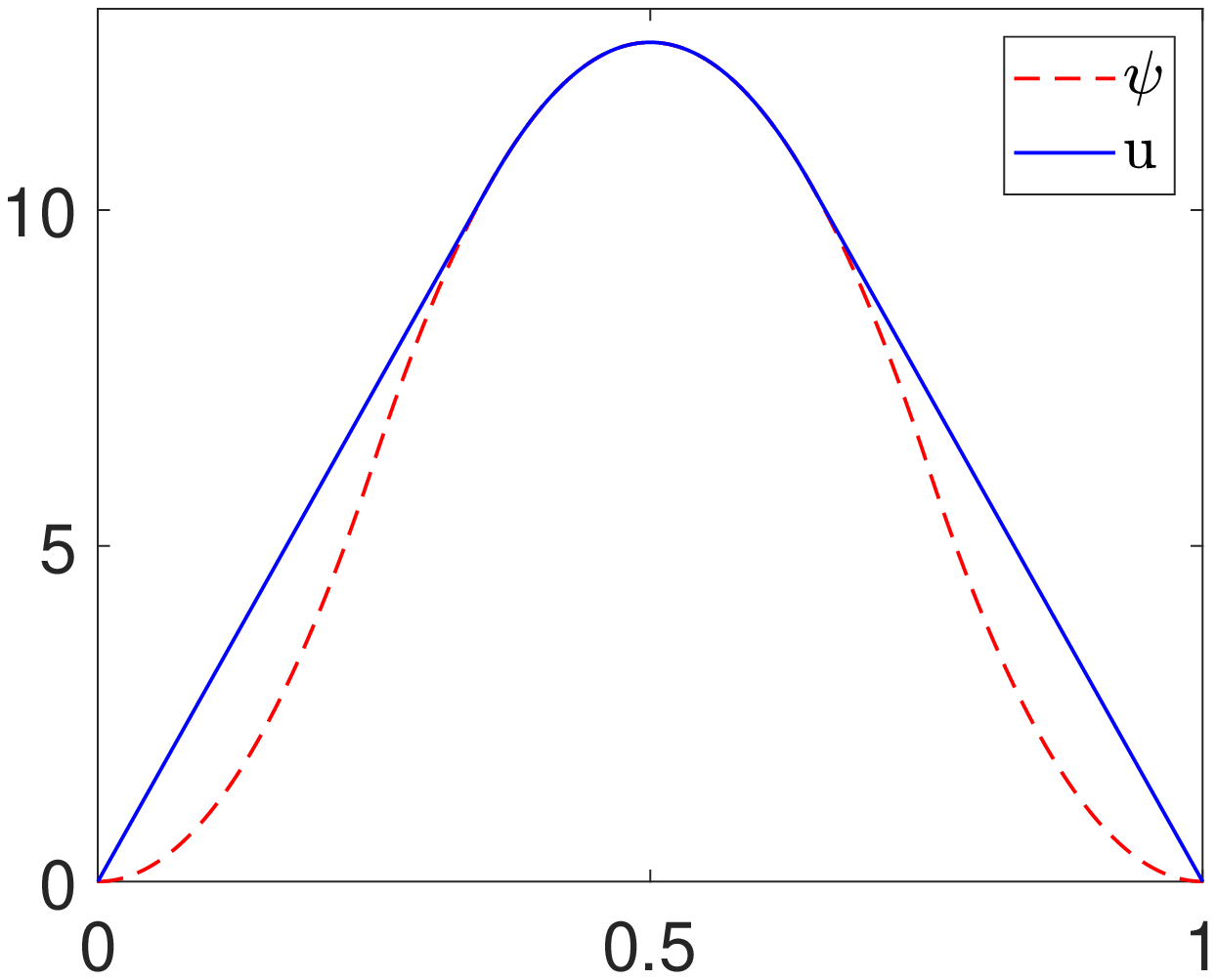} & 
		\includegraphics[width=0.3\textwidth]{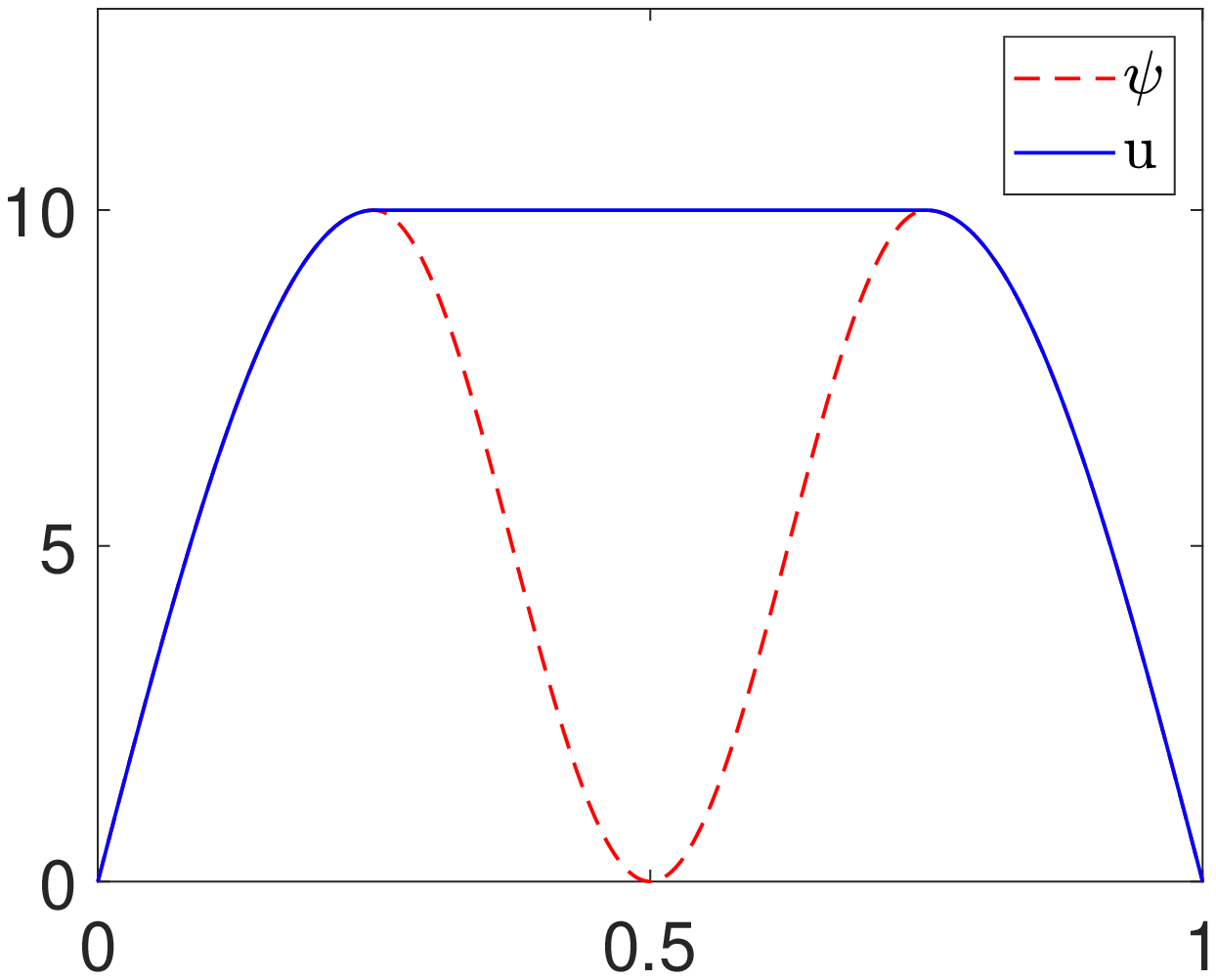} & 
		\includegraphics[width=0.3\textwidth]{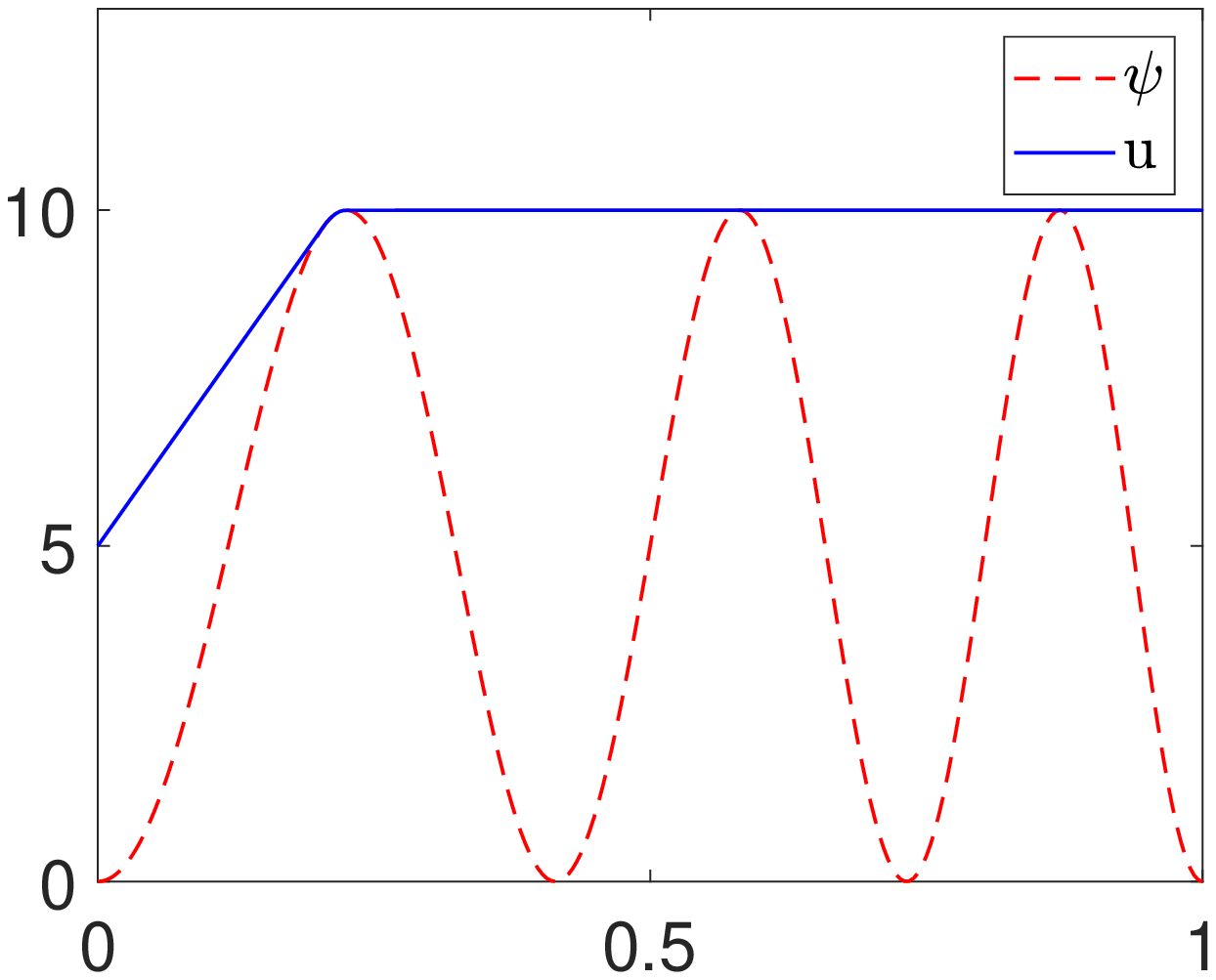}\\
		\includegraphics[width=0.3\textwidth]{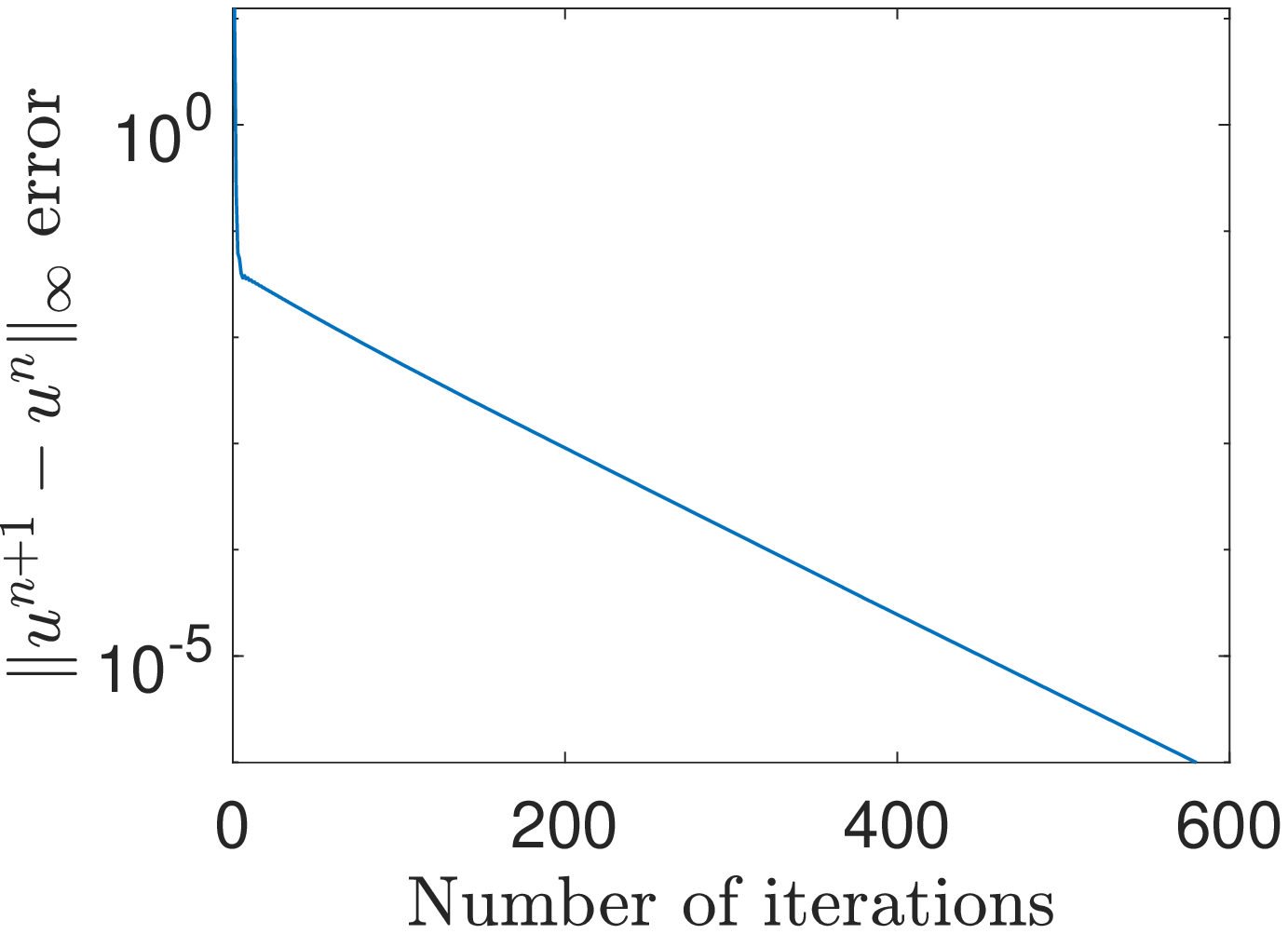} &
		\includegraphics[width=0.3\textwidth]{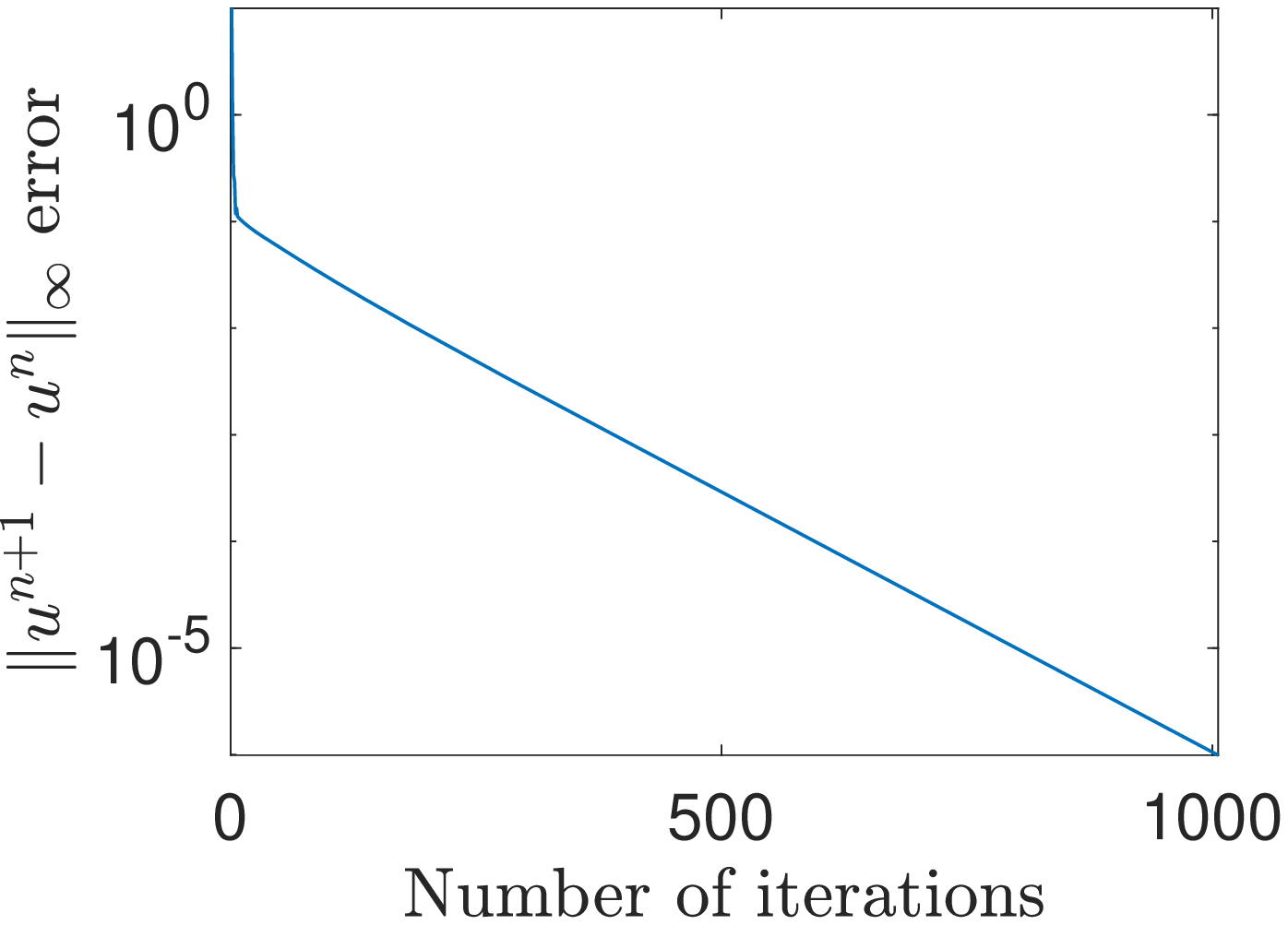} &
		\includegraphics[width=0.3\textwidth]{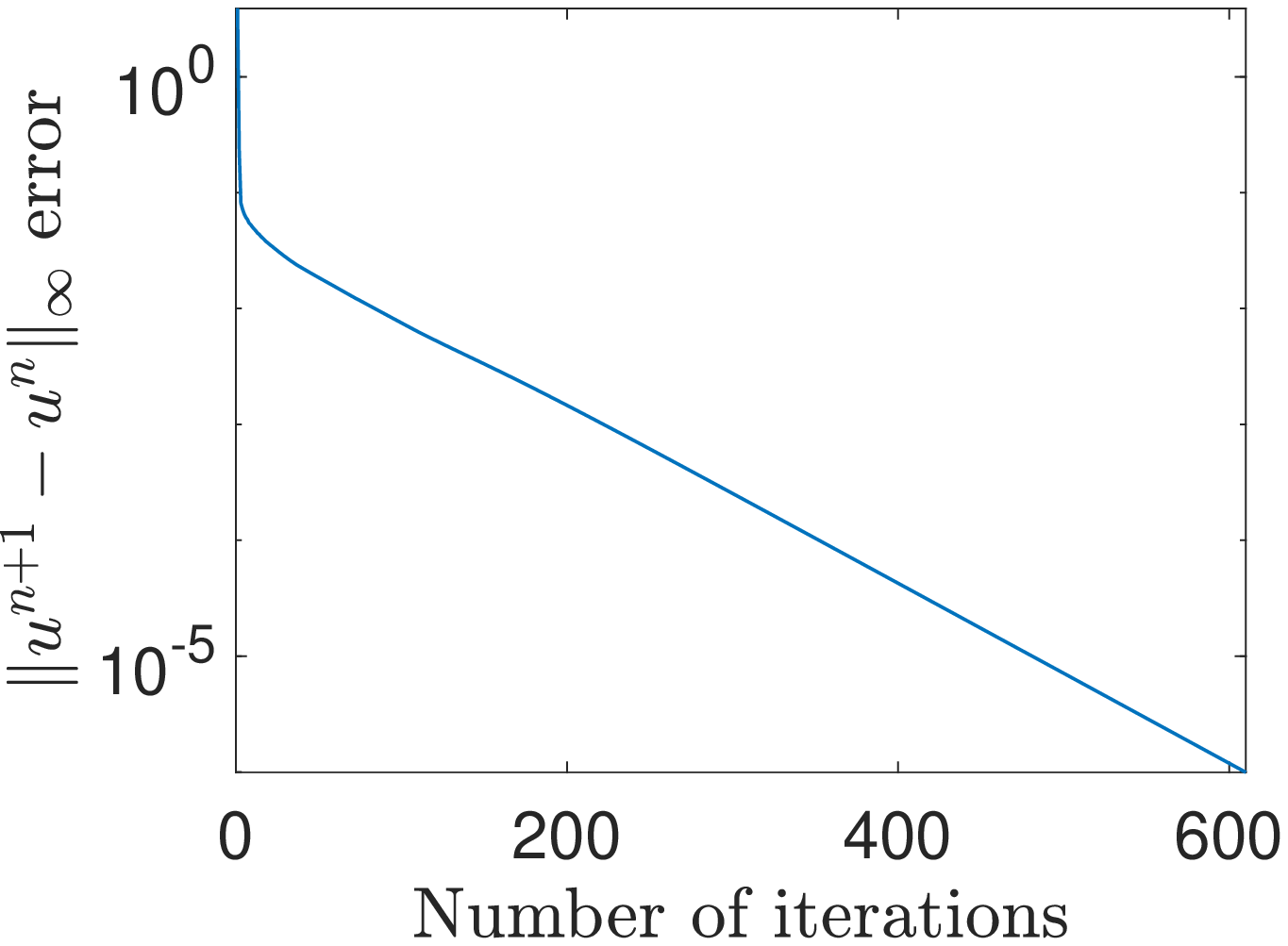} 
	\end{tabular}
	\caption{(One dimensional nonlinear problem with $\Delta x=1/256$.) Column (a)-(c) correspond to obstacles $\psi_1,\ \psi_2,\ \psi_3$, respectively. First row: Graph of the obstacle and the numerical solution by Algorithm \ref{alg.nonlinear}. Second row: Histories of the error $\|u^{n+1}-u^n\|_{\infty}$ .}
	\label{fig.nonlinear1d.general}
\end{figure}

\begin{figure}[t!]
	\centering
	\begin{tabular}{ccc}
		(a) & (b) & (c)\\
		\includegraphics[width=0.3\textwidth]{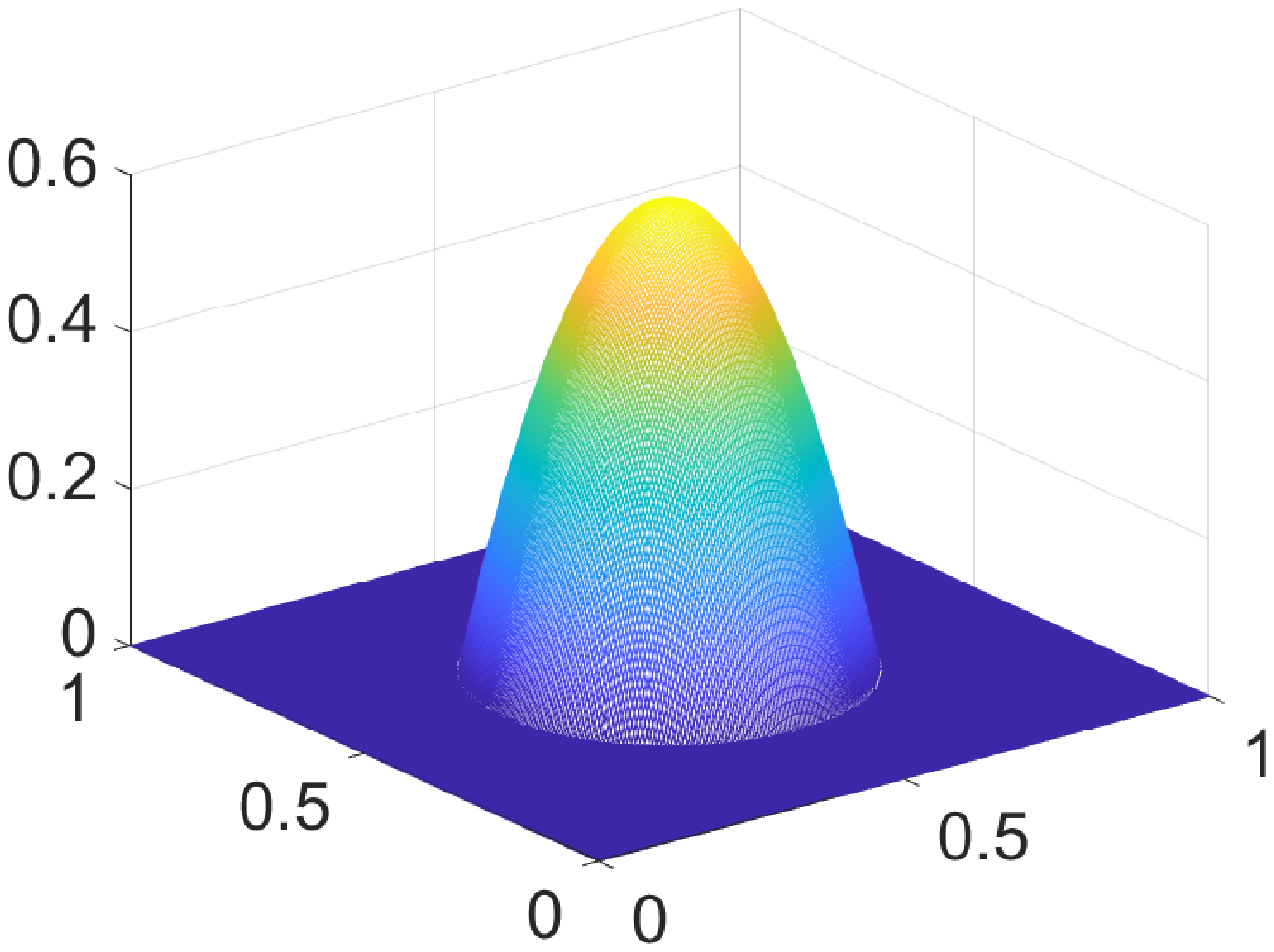} & 
		\includegraphics[width=0.3\textwidth]{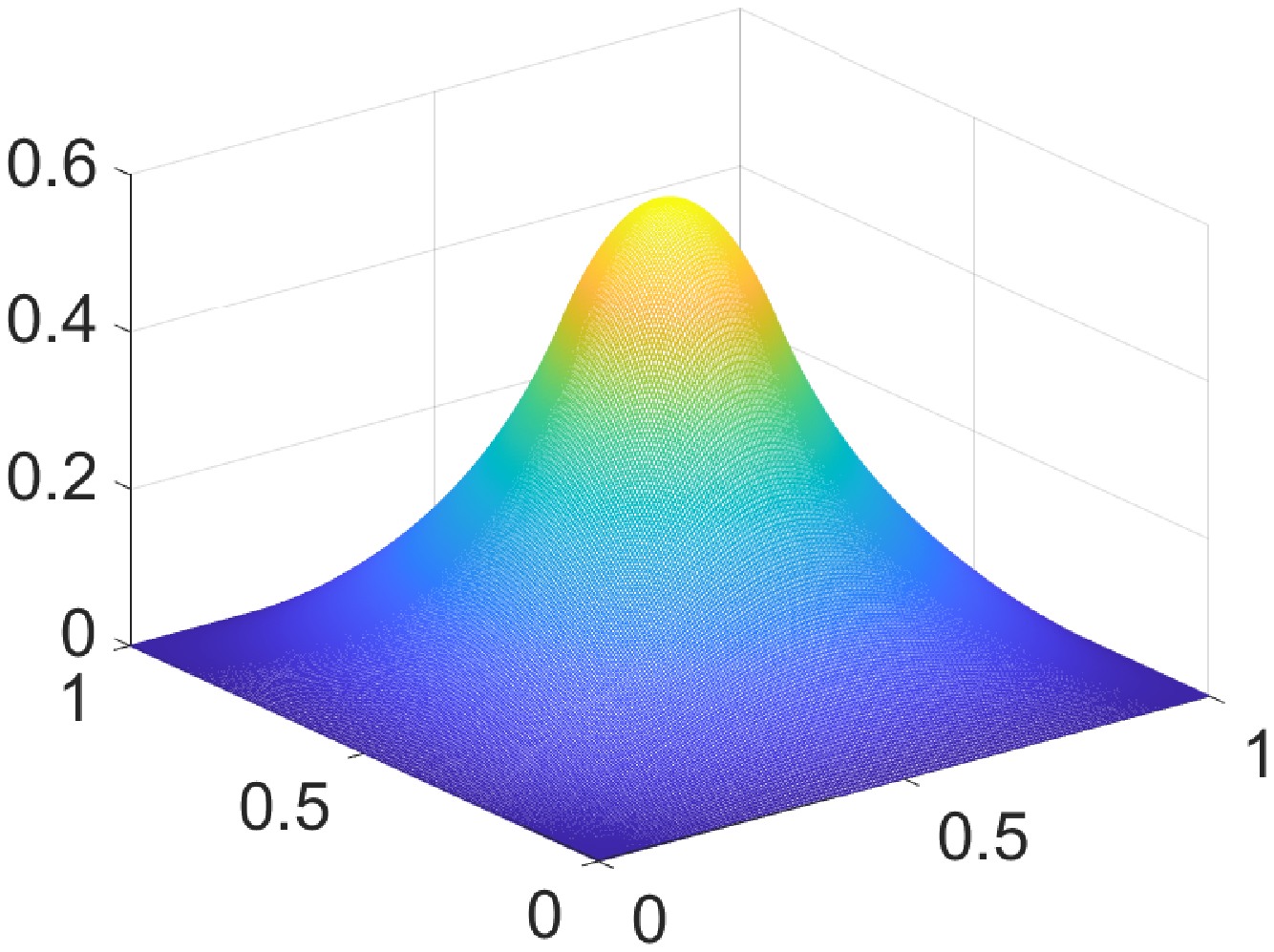} & 
		\includegraphics[width=0.3\textwidth]{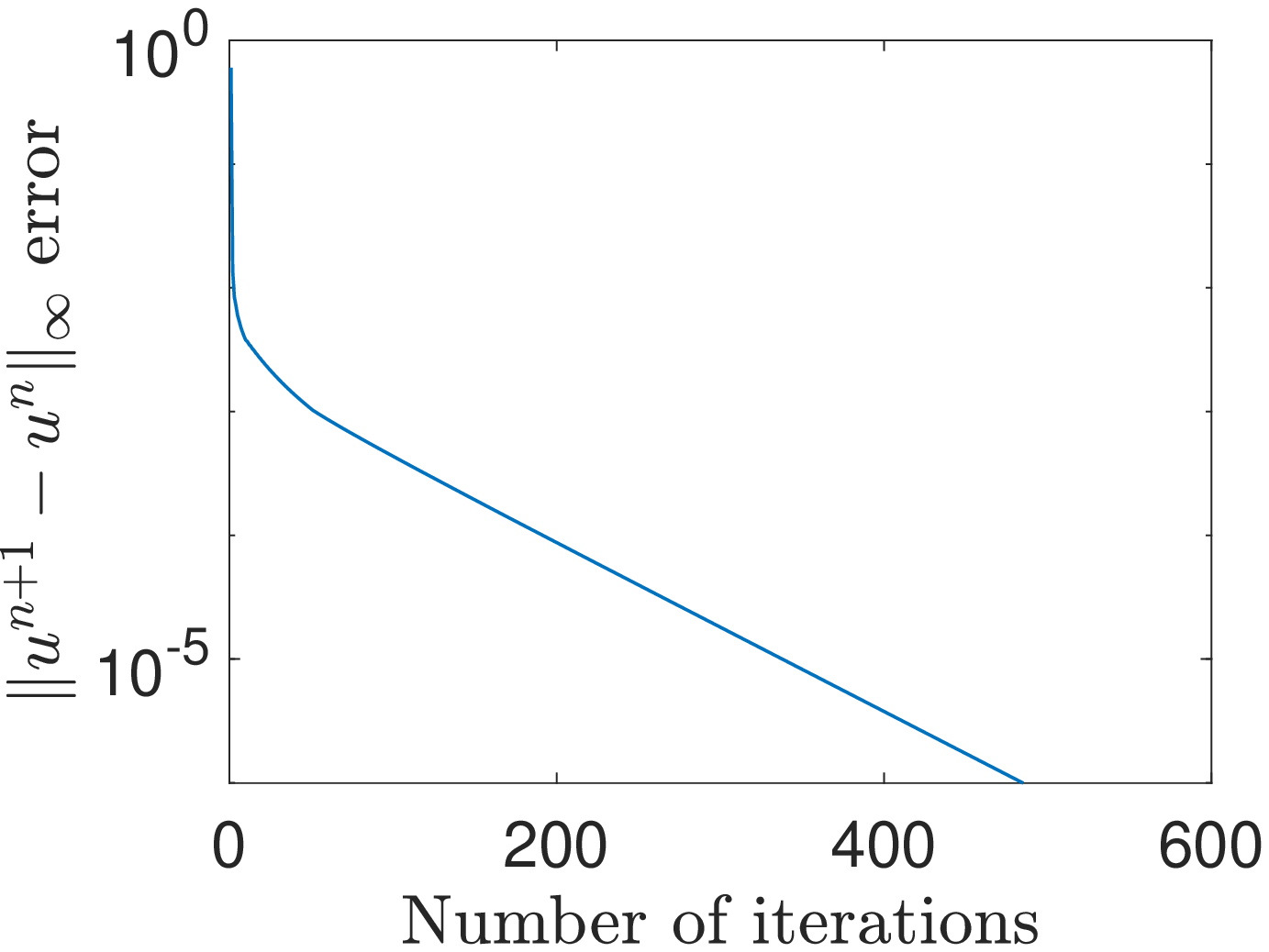}
	\end{tabular}
	\caption{(One dimensional nonlinear problem with $\Delta x=1/256$.) Column (a)-(c) correspond to obstacles $\psi_1,\ \psi_2,\ \psi_3$, respectively. First row: Graph of the obstacle and the numerical solution by Algorithm \ref{alg.nonlinear}. Second row: Histories of the error $\|u^{n+1}-u^n\|_{\infty}$ .}
	\label{fig.nonlinear2d.1}
\end{figure}

\begin{figure}[t!]
\centering
	\begin{tabular}{cc}
		(a) & (b)\\
		\includegraphics[width=0.4\textwidth]{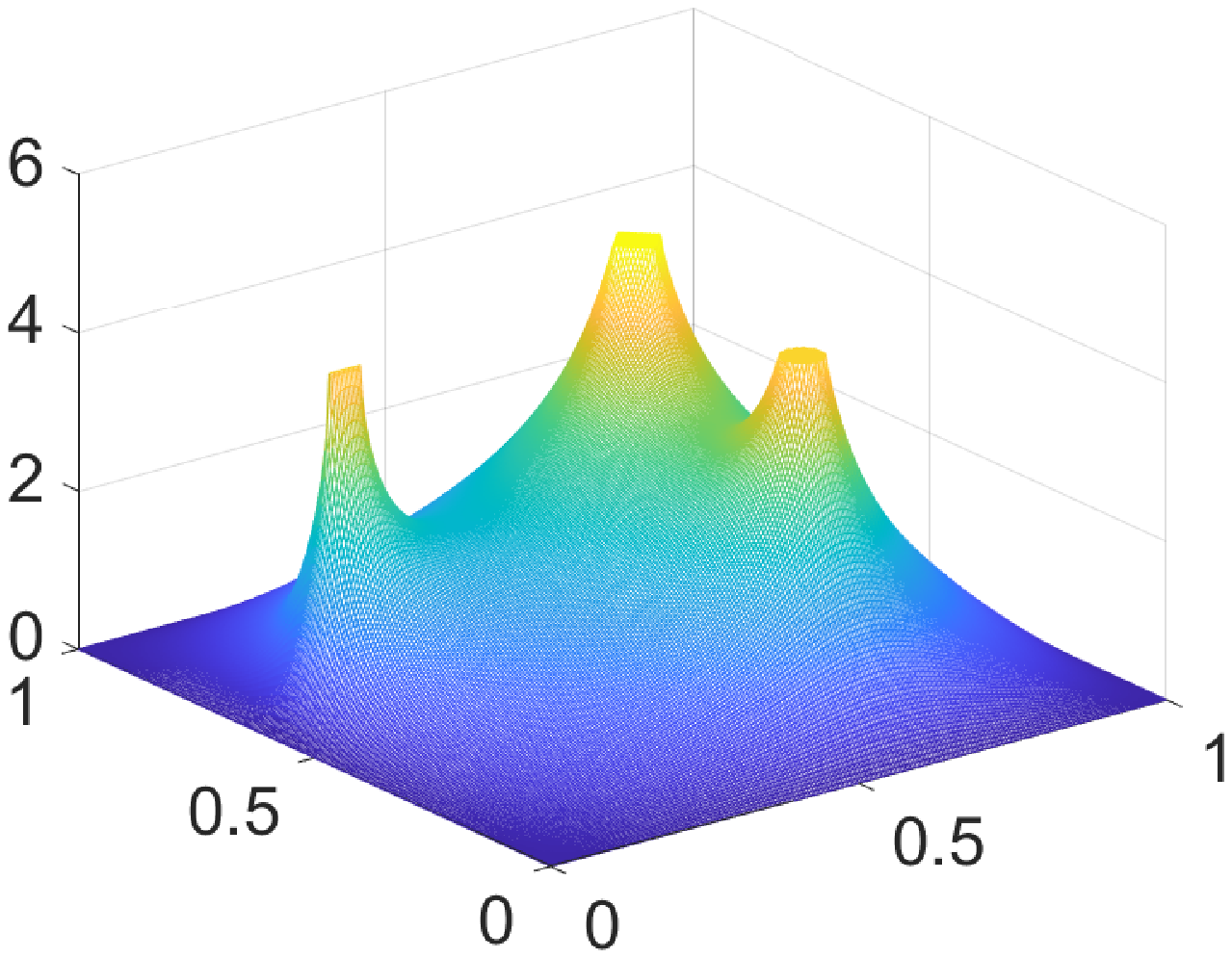} & 
		\includegraphics[width=0.4\textwidth]{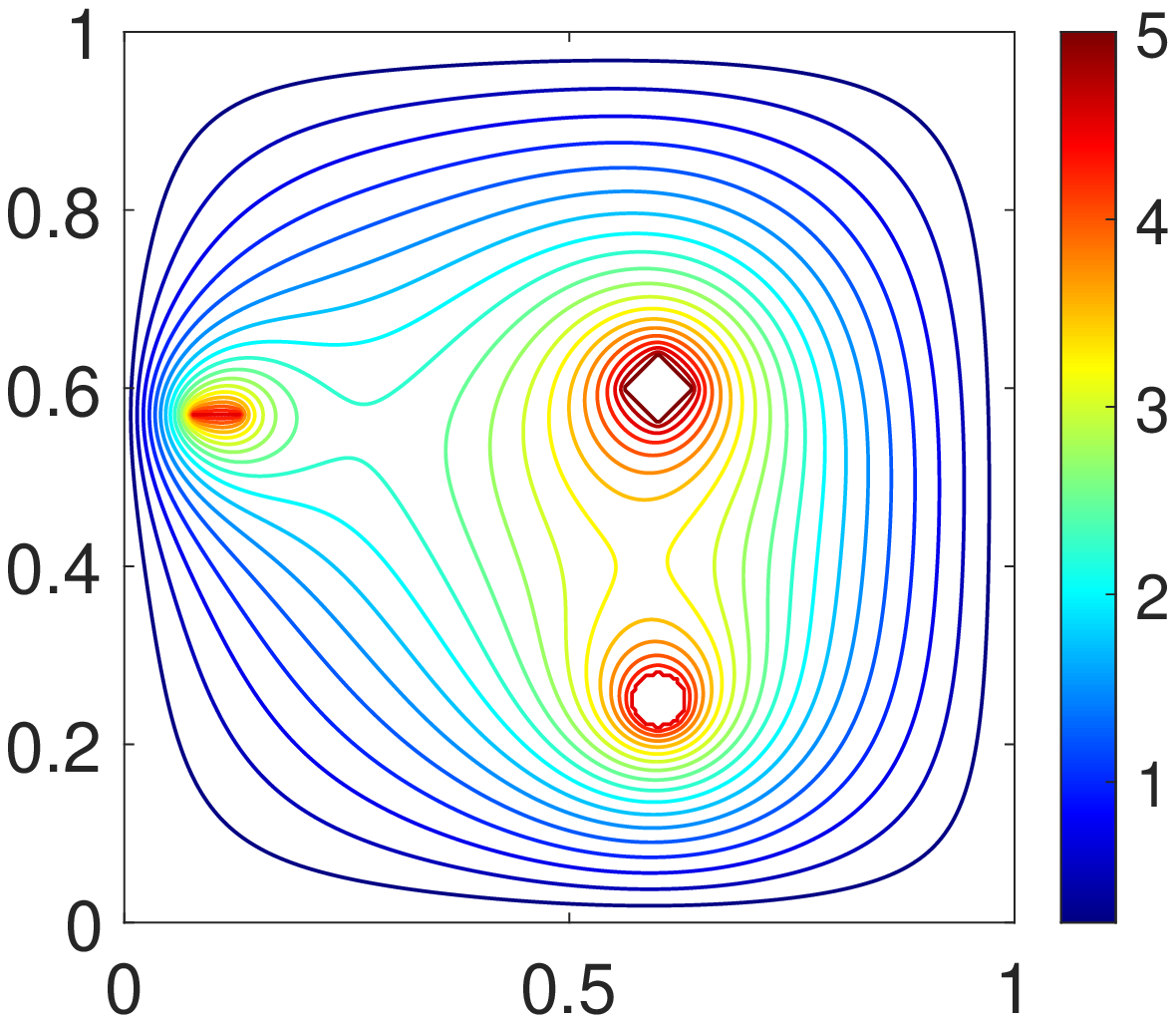}
	\end{tabular}
	\caption{(Two dimensional nonlinear problem with obstacle $\psi_4$ and $\Delta x=1/256$.) (a) Graph of our numerical solution by Algorithm \ref{alg.nonlinear}. (b) Contour of (a).}
	\label{fig.nonlinear2d.2}
\end{figure}

\subsection{Nonlinear obstacle problem} 

We now apply Algorithm \ref{alg.nonlinear} to both one-dimensional and two-dimensional nonlinear problems. For one-dimensional problems, the solutions of nonlinear problems and linear problems are the same. For all experiments in this subsection, we set $\Delta t=10\Delta x, \alpha=0.01$. We test Algorithm \ref{alg.nonlinear} on $\psi_1,\psi_2$ and $\psi_3$ defined in Section \ref{sec.numerical.linear.1D}. We list the results with $\Delta x=1/256$ in Figure \ref{fig.nonlinear1d.general} and the first row shows the obstacles and solutions obtained from Algorithm \ref{alg.nonlinear}. Histories of the error $\|u^{n+1}-u^n\|_{\infty}$ are shown in the second row. Similar to our results of linear problems, the numerical solutions are linear away from the contacting region. Linear convergence is observed.

We test Algorithm \ref{alg.nonlinear} on two two-dimensional problems with computational domain $[0,1]^2$. The first obstacle is defined as
\begin{align}
	\psi_6=\max(0,0.6-8((x-0.5)^2+y-0.5)^2).
\end{align}
With $\Delta x=1/256$, the numerical result is shown in Figure \ref{fig.nonlinear2d.1}. In Figure \ref{fig.nonlinear2d.1}, (a) shows the graph of $\psi_6$, (b) shows the solution obtained from Algorithm \ref{alg.nonlinear}, (c) shows the history of the error $\|u^{n+1}-u^n\|_{\infty}$. For two-dimensional nonlinear problems, we also have linear convergence in $\|u^{n+1}-u^n\|_{\infty}$. 

We then consider obstacle $\psi_4$ defined in (\ref{eq.obs4}). The graph of the numerical solution and its contour are shown in Figure \ref{fig.nonlinear2d.2} (a) and (b), respectively. One can observe that the obtained result is smooth away from the support of the obstacle and contacts the obstacle over the support.

\begin{figure}[t!]
	\centering
	\begin{tabular}{cc}
		(a) & (b)\\
		\includegraphics[width=0.4\textwidth]{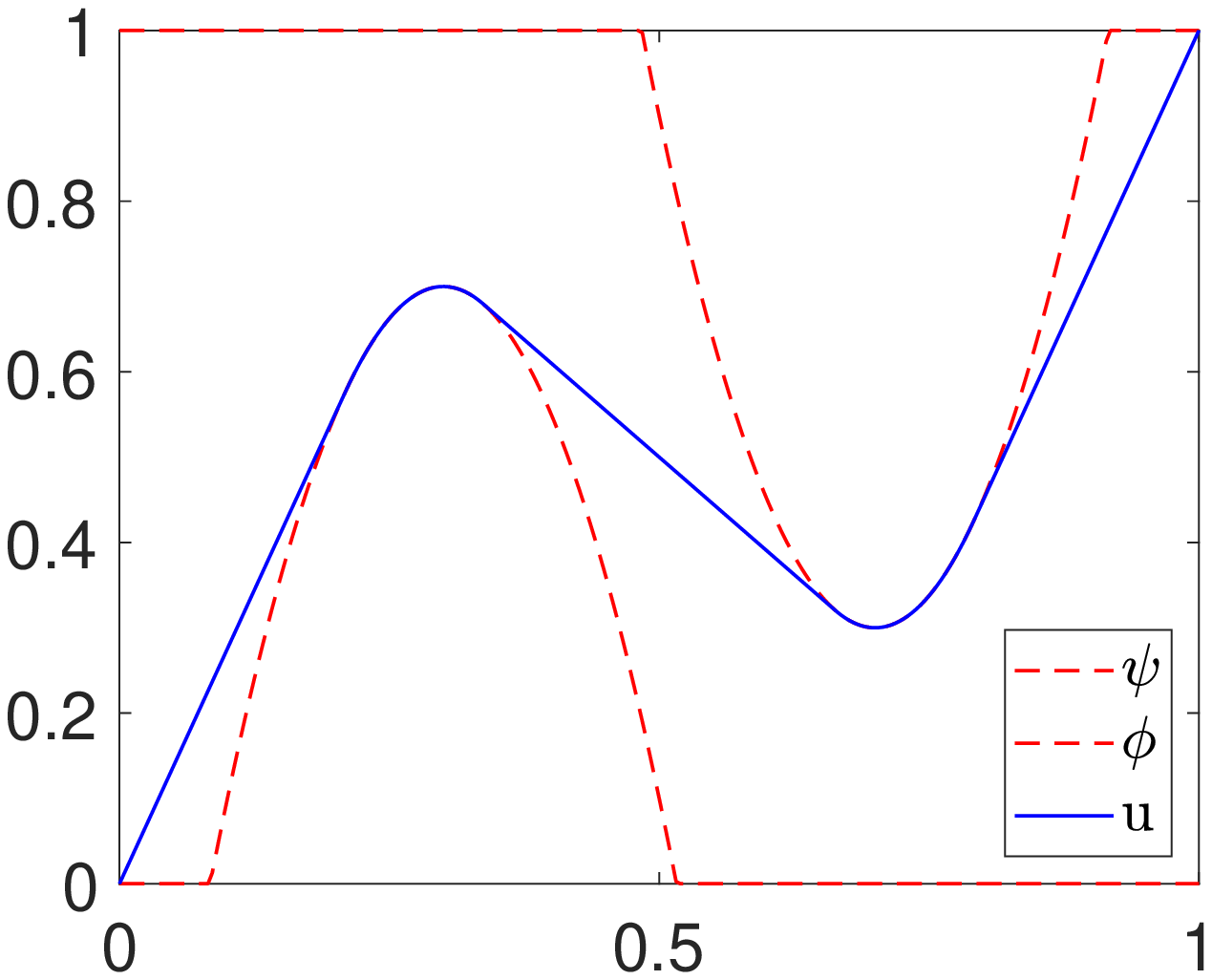} & 
		\includegraphics[width=0.4\textwidth]{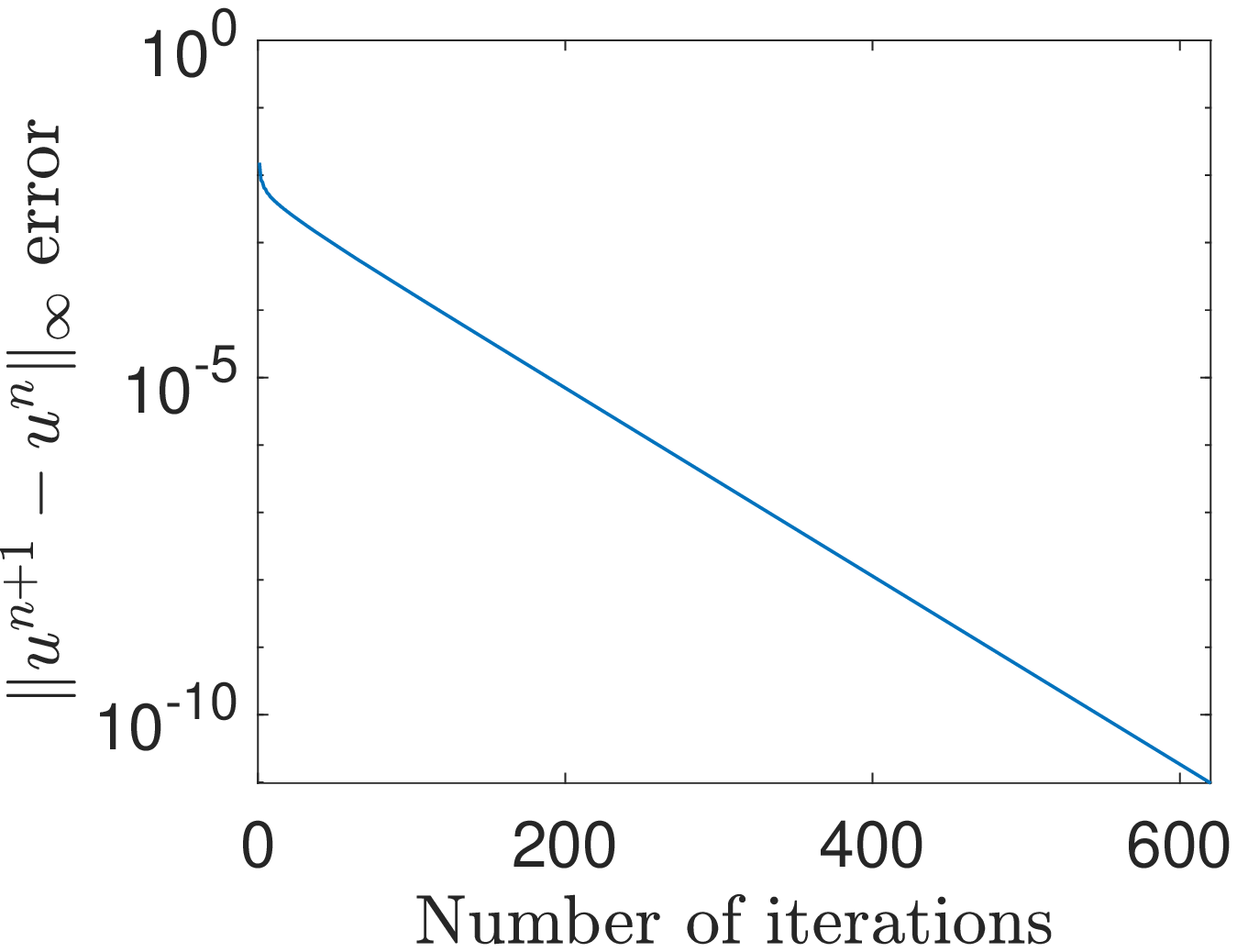}
	\end{tabular}
	\caption{(One-dimensional linear double obstacle problem.) (a) Graph of the obstacles and our numerical solution by the algorithm described in Section \ref{sec.double}. (b) History of the error $\|u^{n+1}-u^n\|_{\infty}$. Linear convergence is observed.}
	\label{fig.double.1D}
\end{figure}

\begin{figure}[t!]
	\centering
	\begin{tabular}{ccc}
		(a) & (b) & (c)\\
		\includegraphics[width=0.3\textwidth]{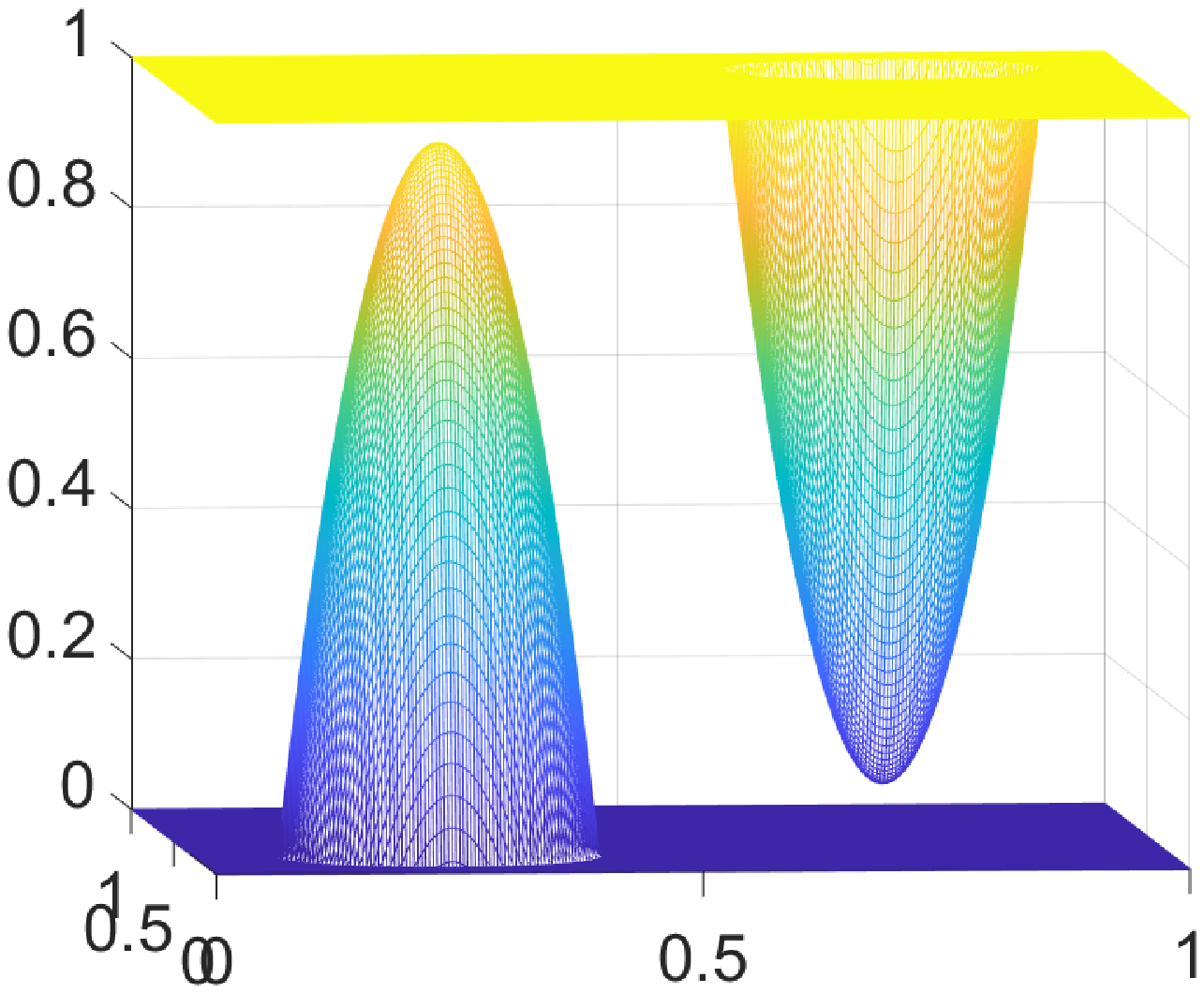} & 
		\includegraphics[width=0.3\textwidth]{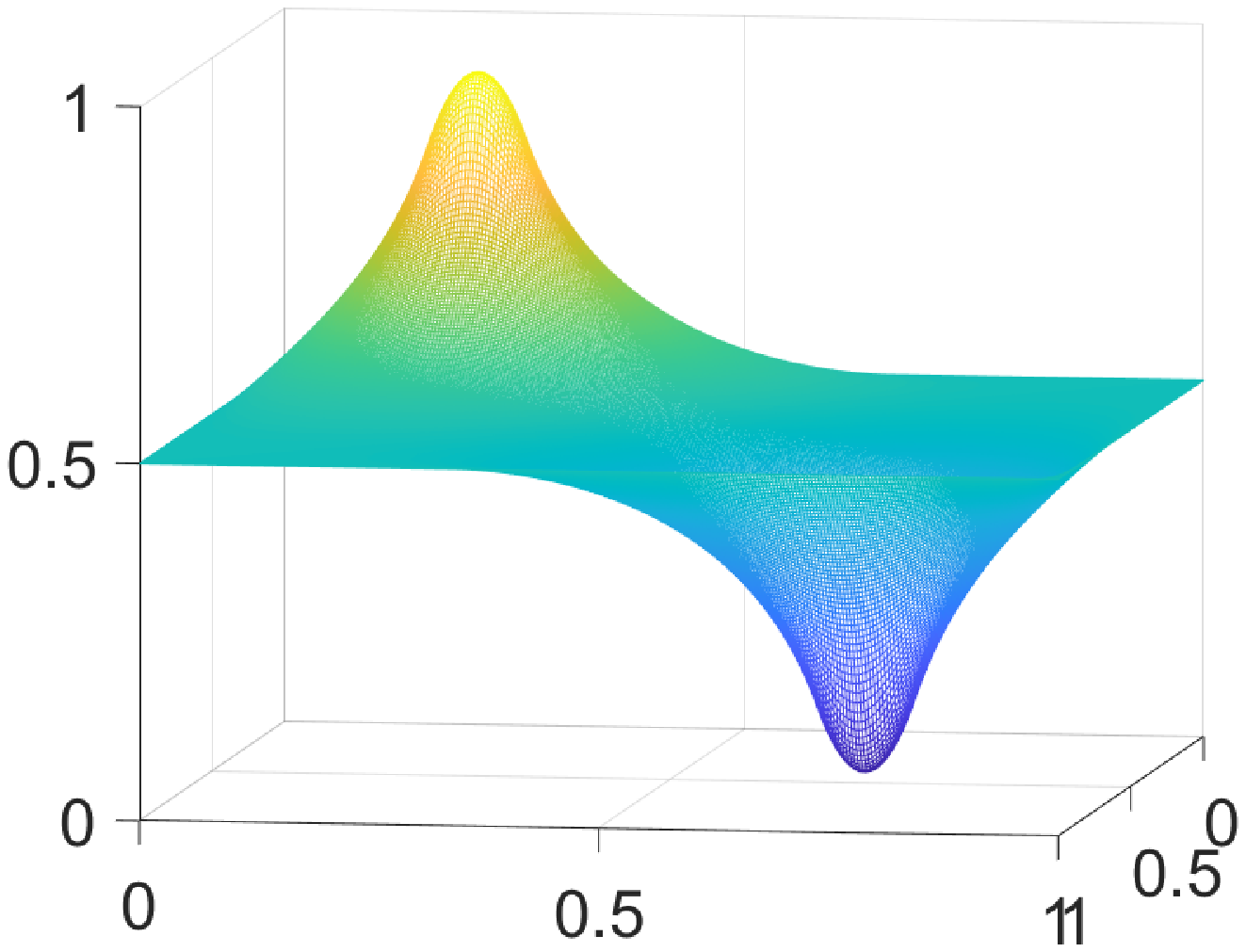}&
		\includegraphics[width=0.3\textwidth]{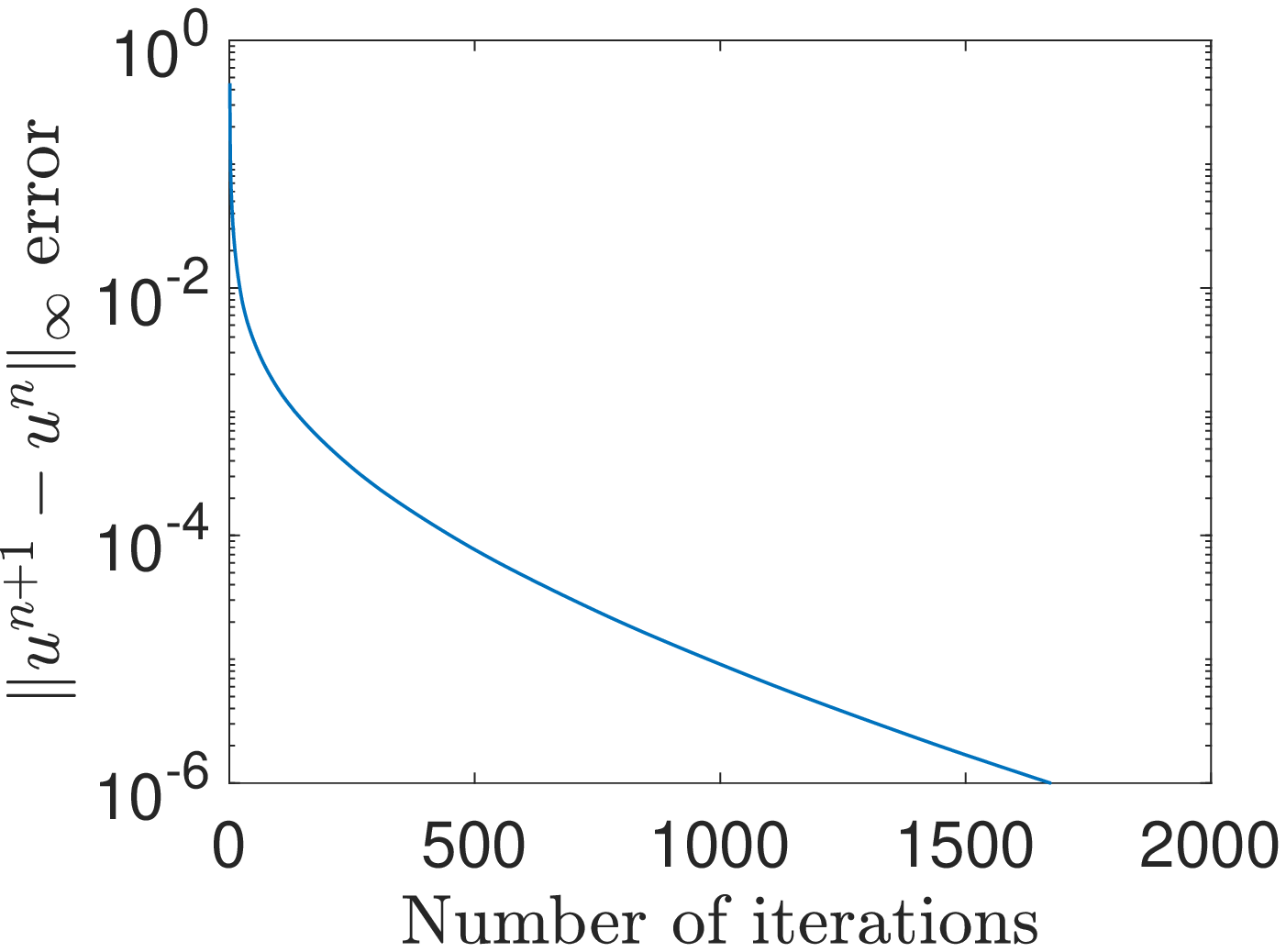}
	\end{tabular}
	\caption{(Two-dimensional linear double obstacle problem.) (a) Graph of the obstacles in (\ref{eq.numerical.double.2D}). (b) Graph of our numerical solution by the algorithm described in Section \ref{sec.double}. (c) History of the error $\|u^{n+1}-u^n\|_{\infty}$.}
	\label{fig.double.2D}
\end{figure}

\subsection{Double obstacle problem}
We test the proposed algorithm for a linear double obstacle problem on the computational domain $[0,1]$. We consider obstacles defined as
\begin{align}
	\begin{cases}
		\psi_7=\max\{0.7-15(x-0.3)^2,0\},\\ 
		\phi_7=\min\{15(x-0.7)^2+0.3,1\}.
	\end{cases}
\end{align}
The boundary condition $u(0)=0$ and $u(1)=1$ is used. The obstacle and our numerical result are shown in Figure \ref{fig.double.1D}. Our solution is linear away from the contacting region. The history of the error $\|u^{n+1}-u^n\|_{\infty}$ is shown in (b), from which linear convergence is observed.

We then consider a two-dimensional linear double obstacle problem. We use computational domain $\Omega=[0,1]^2$ and obstacles
\begin{align}
	\begin{cases}
		\psi_8=\max\{0,0.95-35((x-0.25)^2+(y-0.25)^2)\},\\
		\phi_8=\min\{1,35((x-0.75)^2+(y-0.75)^2)\}.
	\end{cases}
\label{eq.numerical.double.2D}
\end{align}
The obstacles are visualized in Figure \ref{fig.double.2D}(a). With boundary condition $u|_{\partial\Omega}=0.5$ and $\Delta x=1/256, \Delta t=\Delta x$, our numerical solution is shown in (b). The history of the error $\|u^{n+1}-u^n\|_{\infty}$ is shown in (c). 

\begin{figure}[t!]
	\centering
	\begin{tabular}{ccc}
		(a) & (b) & (c)\\
		\includegraphics[width=0.3\textwidth]{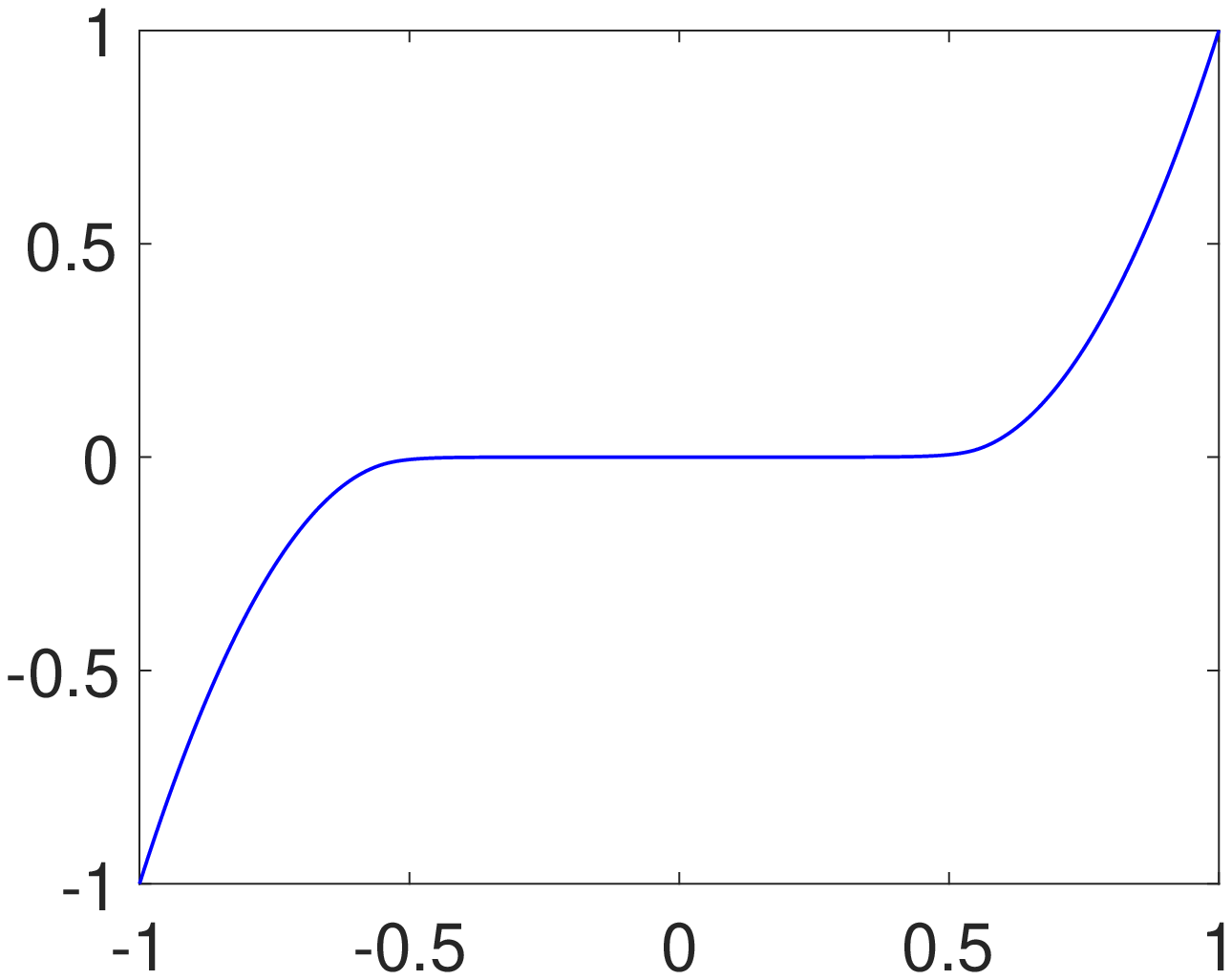} & 
		\includegraphics[width=0.3\textwidth]{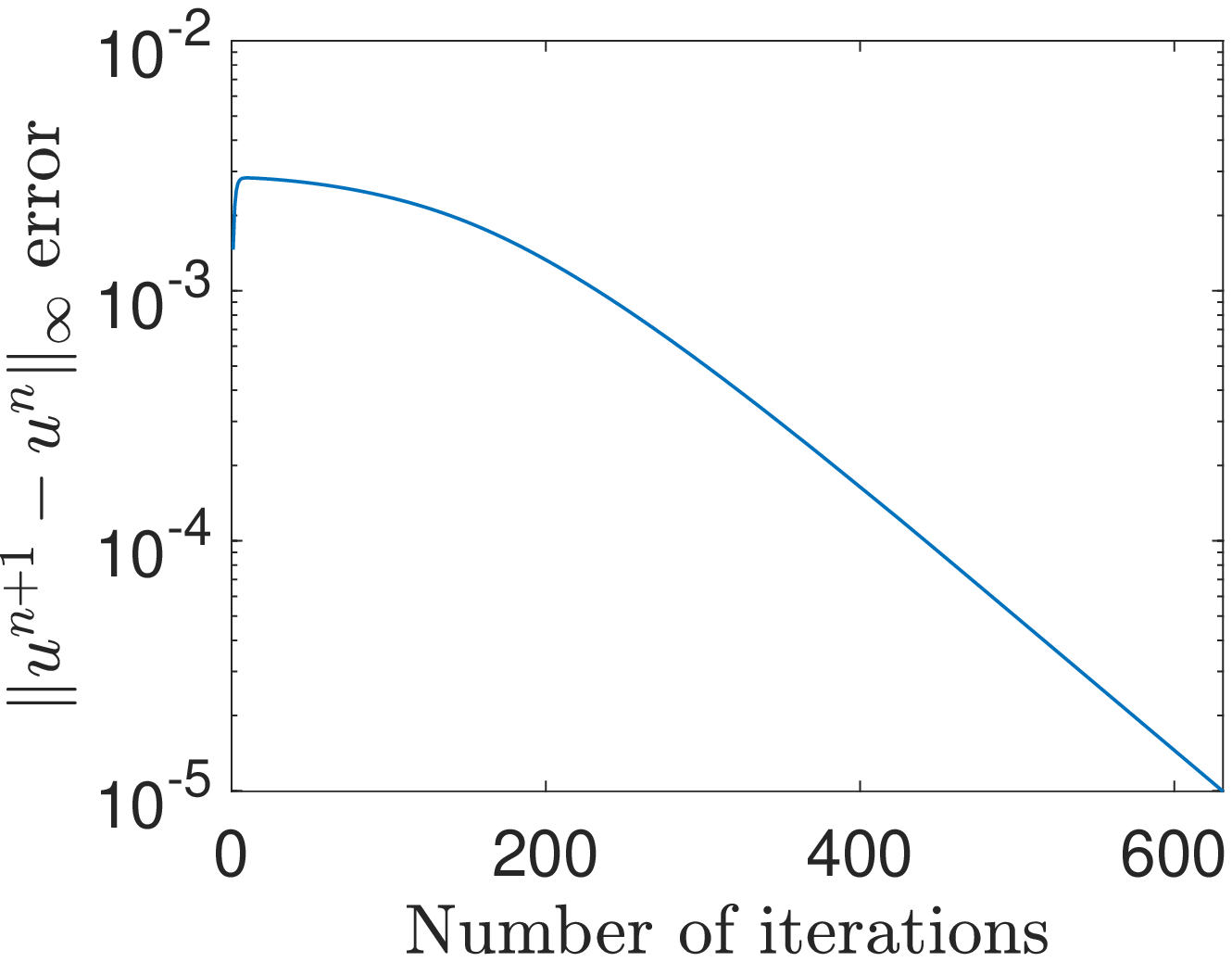}&
		\includegraphics[width=0.3\textwidth]{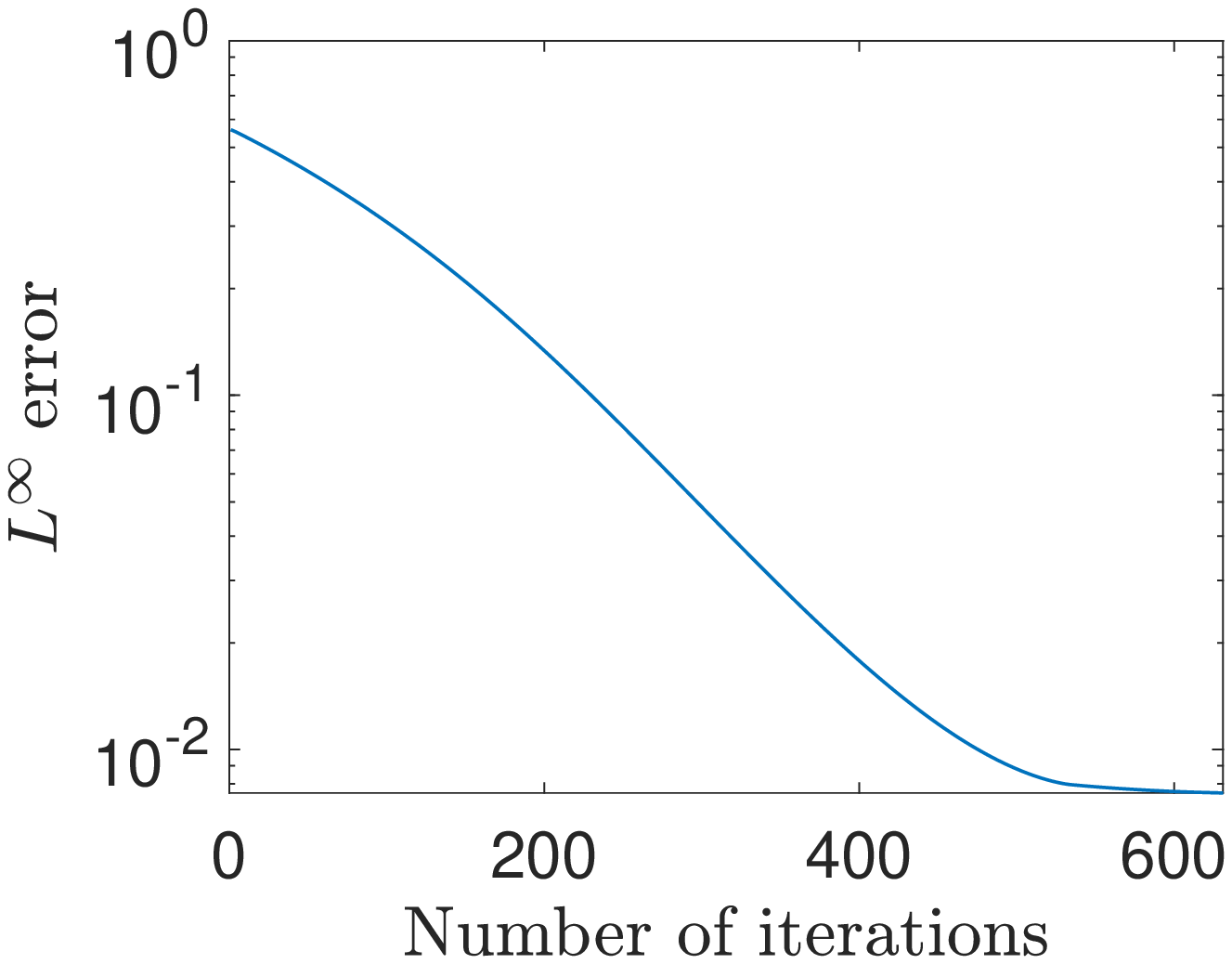}
	\end{tabular}
	\caption{(One-dimensional two-phase membrane problem with $\mu_1=\mu_2=8$.) (a) Graph of the numerical solution by the algorithm described in Section \ref{sec.twophase}. (b) History of the error $\|u^{n+1}-u^n\|_{\infty}$. (c) History of the $L^{\infty}$ error.}
	\label{fig.twophase.1D}
\end{figure}

\begin{figure}[t!]
	\centering
	\begin{tabular}{c}
		\includegraphics[width=0.4\textwidth]{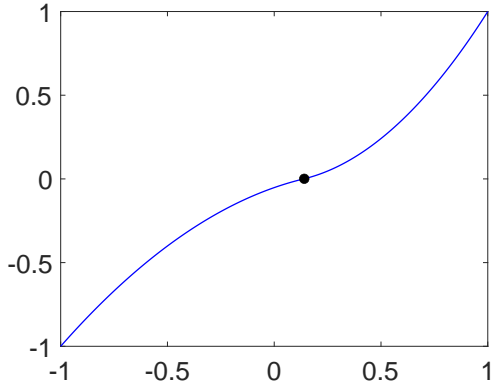} 
	\end{tabular}
	\caption{(One-dimensional two-phase membrane problem with $\mu_1=2,\mu_2=1$.) Graph of the numerical solution by the algorithm described in Section \ref{sec.twophase}. The black dot denotes the free boundary. }
	\label{fig.twophase1.1D}
\end{figure}

\subsection{Two-phase membrane problem}
We test the proposed algorithm discussed in Section \ref{sec.twophase} for the two-phase membrane problem. For all experiments in this subsection, we use the computation domain $[-1,1]$, $\Delta x=2/256, \Delta t=0.1\Delta x$ and $\alpha=500$. We first consider a symmetric case where $\mu_1=\mu_2=8$. For this problem, the exact solution is given as
\begin{align}
	u^*=\begin{cases}
		-4x^2-4x-1 &\mbox{ for } -1\leq x\leq -0.5,\\
		0 &\mbox{ for } -0.5\leq x \leq 0.5,\\
		4x^2-4x+1 &\mbox{ for } 0.5\leq x\leq 1.
	\end{cases}
\end{align}
Our numerical solution is shown in Figure \ref{fig.twophase.1D}(a). Histories of the error $\|u^{n+1}-u^n\|_{\infty}$ and the $L^{\infty}$ error are shown in (b) and (c), respectively. The proposed algorithm converges within 600 iterations. Linear convergence is observed for the solution to converge.

We next consider a non-symmetric setting with $\mu_1=2,\mu_2=1$. The numerical solution is shown in Figure \ref{fig.twophase1.1D}, in which the black dot denotes the calculated free boundary. From our numerical solution, the location of the calculated free boundary is approximately $x=0.141$, which matches the observation in \cite{tran20151}.

\section{Conclusion}\label{sec.conclusion}
We proposed two CADE based operator-splitting methods for the linear and nonlinear obstacle problems. For each problem, we remove the constraint by introducing an indicator function and convert the unconstrained optimization problem to finding the steady state solution of an initial value problem. The initial value problem is then time--discretized by the operator--splitting method so that each subproblem either has an explicit solution or can be solved efficiently. After splitting, one sub-step needs to solve a heat equation with obstacles, for which we proposed the CADE algorithm, a fully explicit algorithm. The proposed methods can be easily extended to solve other free boundary problems, such as the double obstacle problem and the two-phase membrane problem.
The performance of the proposed methods are demonstrated by systematic numerical experiments. Compared to other existing methods, our methods are more efficient while keeping similar accuracy.

\bibliographystyle{abbrv}
\bibliography{ref}

\end{document}